\documentclass[a4paper]{article}

\usepackage{amsthm}

\usepackage{appendix}
\usepackage[all]{xy}\usepackage[latin1]{inputenc}        
\usepackage[dvips]{graphics,graphicx}
\usepackage{amsfonts,amssymb,amsmath,color,mathrsfs, amstext}
\usepackage{amsbsy, amsopn, amscd, amsxtra, amsthm,authblk, enumerate}
\usepackage[shortlabels]{enumitem}
\usepackage{listings}
\usepackage{upref}
\usepackage[colorlinks,
            linkcolor=red,
            anchorcolor=red,
            citecolor=red
            ]{hyperref}

\usepackage{geometry}
\usepackage[displaymath]{lineno}
\usepackage{float}
\usepackage{xcolor}
\usepackage{algorithmic,algorithm}

\usepackage{subeqnarray}
\usepackage{cases}
\usepackage{eqnarray}

\usepackage{yhmath}

\usepackage[normalem]{ulem}

\numberwithin{equation}{section}     

\def\b{\boldsymbol}

\def\cS{\mathcal{S}}
\def\R{\mathbb{R}}
\newcommand{\E}{\mathbb{E}}

\newtheorem*{theorem*}{Theorem}

\newtheorem{theorem}{Theorem}[section]
\newtheorem{lemma}{Lemma}[section]
\newtheorem{assumption}{Assumption}[section]
\newtheorem{proposition}{Proposition}[section]
\newtheorem{remark}{Remark}[section]

\lstdefinestyle{Python}{
    language        =   Python, 
    basicstyle      =   \zihao{-5}\ttfamily,
    numberstyle     =   \zihao{-5}\ttfamily,
    keywordstyle    =   \color{blue},
    keywordstyle    =   [2] \color{teal},
    stringstyle     =   \color{magenta},
    commentstyle    =   \color{red}\ttfamily,
    breaklines      =   true,   
    columns         =   fixed,  
    basewidth       =   0.5em,
}

\begin{document}

\title{A modified tamed scheme for stochastic differential equations with superlinear drifts}
\author[a]{Zichang Ju\thanks{E-mail: jzc\_sypqf@sjtu.edu.cn}}
\author[b]{Lei Li\thanks{E-mail: leili2010@sjtu.edu.cn}}
\author[a]{Yuliang Wang\thanks{E-mail: YuliangWang\_math@sjtu.edu.cn}}
\affil[a]{School of Mathematical Sciences, Institute of Natural Sciences, Shanghai Jiao Tong University, Shanghai, 200240, P.R.China.}
\affil[b]{School of Mathematical Sciences, Institute of Natural Sciences, MOE-LSC, Shanghai Jiao Tong University, Shanghai, 200240, P.R.China.}
\date{}
\maketitle

\begin{abstract}
Explicit discretizations of stochastic differential equations often encounter instability when the coefficients are not globally Lipschitz. The truncated schemes and tamed schemes have been proposed to handle this difficulty, but truncated schemes involve analyzing of the stopping times while the tamed schemes suffer from the reduced order of accuracy.
We propose a modified tamed scheme by introducing an additional cut-off function in the taming, which enjoys the convenience for error analysis and preserving the original order of explicit discretization. While the strategy could be applied to any explicit discretization, we perform rigorous analysis of the modified tamed scheme for the Euler discretization as an example. Then, we apply the modified tamed scheme to the stochastic gradient Langevin dynamics for sampling with super-linear drift, and obtain a uniform-in-time near-sharp error estimate under relative entropy.
\end{abstract}

\section{Introduction}\label{sec:intro}
       
Stochastic differential equations (SDEs) have a wide range of applications in various related fields, including financial engineering, artificial intelligence, computational chemistry, to name a few. Though admitting well-posedness under suitable conditions \cite{rogers2000diffusions}, most SDEs do not have an explicit analytical solution. Therefore, seeking efficient and stable numerical approaches to solve SDEs is of great significance. One most classical and simplest method is the celebrated Euler-Maruyama scheme. It has been verified both theoretically and experimentally that the Euler's scheme converges to the original SDE with at least half-order strong error when both the drift and diffusion terms are globally Lipschitz \cite{gikhman2004theory, kloeden1992stochastic}. However, when the drift grows super-linearly, it can be shown that the moments of Euler approximation could diverge to infinity \cite{mattingly2002ergodicity,milstein2005numerical,hutzenthaler2011strong}. Some schemes modify the drift term by truncating or taming to guarantee the moment bound, and thus one is able to derive the convergence rigorously \cite{hutzenthaler2012strong, sabanis2013note, mao2015truncated, guo2017partially}. Despite the efficiency of these schemes, it turns out that the rate of their convergence (especially the weak convergence rate) is not the most satisfactory (see for instance \cite[Theorem 5.6]{wang2024weak}).

 Let us consider the following SDE that allows time dependency of both the drift and diffusion:
        \begin{equation}\label{SDE}
		d{X}(t)= {b}\big(t, {X}(t)\big)dt+ {{\sigma}}\big(t, {X}(t)\big)d {W}_t,\quad X(0) = X_0.
	\end{equation}
	Here $X(t) \in \mathbb{R}^d$, $b: \mathbb{R}_{+}\times \mathbb{R}^d \rightarrow \mathbb{R}^d$ is a non-global-Lipschitz drift function, $\sigma : \mathbb{R}_{+} \times \mathbb{R}^{d} \rightarrow \mathbb{R}^{d\times d'}$ is a global-Lipschitz diffusion function, and $W_t$ is the $d'$-dimensional Brownian motion on a filtered probability space $(\Omega, \mathcal{F}, \mathbb{P}, (\mathcal{F}_t)_{t\geq 0})$.  For a given constant time step size $h$, denote $T_i := ih$, $i \in \mathbb{N}$. We further denote 
    \begin{gather}
    \kappa(t) := 
    \begin{cases} 
    T_k & t \in (T_k, T_{k+1}]\\
    0  & t=0.
    \end{cases}
    \end{gather}

    Since the classical Euler's method may blow up when the drift grows super-linearly, it is a natural idea to make some modifications to the drift. One typical approach is the tamed Euler scheme \cite{hutzenthaler2012strong,sabanis2013note}:
    \begin{equation}\label{eq:originaltamed}
		d {X_h}(t)= {b^h_{\mathrm{tame}}}\Big(\kappa(t), {X_h}\big(\kappa(t)\big)\Big)dt+ {\sigma}\Big(\kappa(t), {X_h}\big(\kappa(t)\big)\Big)d {W}_t,\quad X_h(0) = X_0,
	\end{equation} 
	where
	\begin{equation*}
		 {b^h_{\mathrm{tame}}}(t, {x}):=\frac{ {b}(t, {x})}{1+h^{\alpha}| {b}(t, {x})|},
	\end{equation*}
	with $\alpha \in (0,1/2]$. Obviously, the tamed drift if bounded by $|b^h_{\mathrm{tame}}| \leq \min(h^{-\alpha}, |b|)$ for any $t$ and $x$. Based on this property, it can be rigorously proved that the strong error of \eqref{eq:originaltamed} is the same as the original Euler's scheme under the global Lipschitz condition, which is half order (see \cite[Corollary 2.3]{sabanis2013note}). However, its weak order can not achieve as high as the original discretization (at most $\alpha$), mainly due to the error brought by the taming \cite[Lemma 4.1]{sabanis2013note}. In many applications for SDEs, especially those for sampling tasks, people tend to focus more on the error under weaker topologies. This is because with the samples generated by the numerical scheme, many practical tasks often require the further computation of some statistical quantities via integration with respect to some given function. 

    Another category of numerical schemes handling the super-linear growing drift applies the ``truncating" idea. For the time-homogeneuous case ($b(t,x) = b(x)$, $\sigma(t,x) = \sigma(x)$), in \cite{mao2016convergence, mao2015truncated}, the truncated Euler's scheme is given by
    \begin{equation}\label{eq:truncatedEuler}
        dX_{h}(t) = b_{\mathrm{trun}}^{h}( X_{h}(\kappa(t))) dt + \sigma_{\mathrm{trun}}^{h}( X_{h}(\kappa(t))) dW_t,
    \end{equation}
    where for $h \in(0,1)$, the truncated function is defined as ($\sigma_{\mathrm{trun}}^{h}$ is similar)
    \begin{equation*}
b_{\mathrm{trun}}^{h}(x)=b\left(\left(|x| \wedge \mu^{-1}(\epsilon(h))\right) \frac{x}{|x|}\right).
    \end{equation*}
Clearly, $b_{\mathrm{trun}}^{h}$ agrees with $b$
for $|x|\le \mu^{-1}(\epsilon(h))$.
The truncated scheme can also guarantee the stability of the numerical solution as the far away field has been discarded. The truncated scheme usually does not have the reduced order problem. However, its rigorous convergence often involves analyzing of the stopping times, thus bringing potential challenges in various practical settings especially for interacting particle systems (see, for example, \cite{guo2024convergence}).

In this paper, we propose a modified tamed scheme, mainly by introducing a cut-off function in the taming. The idea is to keep the taming only for large $|x|$, while discarding the taming for $|x|=O(1)$, which does not affect the numerical stability and contributes to the typical values of the solutions (see Section \ref{sec:scheme} for more details).  This approach in spirit is to combine the truncating and taming ideas together, avoiding the disadvantages of both-- we no longer need to analyzing the stopping times and are albe to keep the order of accuracy.

This modified scheme in principle can be applied to any explicit discretization, and we use the Euler discretization as an example. 
In the corresponding modified tamed Euler (MTE) scheme (see more details in \eqref{MTE} below) and the particular version for Langevin sampling (tamed stochastic gradient Langevin dynamics, T-SGLD), 
we justify the convergence of the proposed scheme both theoretically and numerically in Sections \ref{sec:convergence} - \ref{sec:numerical} below. 
As we will discuss in more detail in Section \ref{sec:scheme} below, our proposed MTE scheme also allows inaccurate drift (i.e. the random batch approximation for the original drift) to reduce the computational cost. Note that the theoretical analysis for schemes with random batch is usually non-trivial, especially when it comes to obtaining sharp error bounds under relative entropy \cite{li2022sharp, huang2025mean}. In this paper, we successfully overcome the additional challenges brought by the random batch and obtain a quantitative error bound with a nearly optimal rate.

\paragraph{Related work} 
The ``taming" idea has also been applied to construct numerical schemes under other settings. For instance, the tamed Milstein method in \cite{kumar2016tamed, wang2013tamed}, the tamed scheme for non-Lipschitz diffusion coefficients in \cite{sabanis2016euler}, to name a few. In \cite{wang2024thesis}, the tamed scheme has been extended to inaccurate drifts (random batch approximation) with supelinear growth. Recently, the manuscript \cite{liu2025geometric} has proposed a class of new tamed schemes for stochastic PDEs. Besides the taming idea, there exist other approaches to handle the approximation for SDEs with non-globally Lipschitz drifts. In various truncated schemes \cite{mao2015truncated,deng2019truncated,liu2020truncated}, the values of drift and diffusion functions in the far field are modified so that the resulting scheme has strong convergence to the original SDE under the super-linear growth condition. In \cite{fang2020adaptive}, an adapted Euler's scheme was proposed, where the timestep is adjusted according to the value of the SDE solution and the growing behavior of the drift function. In \cite{bossy2021weak}, for a class of SDEs with non-Lipschitz drifts and multiplicative noises, the authors proposed a (semi-explicit) exponential-Euler scheme with rigorous weak convergence. Remarkably, for many schemes mentioned above, recently in \cite{wang2024weak}, the authors give some general theorems for their weak error analysis, which partially motivates us to design novel schemes with improved weak convergence rates.


The rest of this paper is organized as follows. We introduce the details of the proposed modified tamed scheme in Section \ref{sec:scheme}. The convergence analysis of the MTE is provided in section \ref{sec:convergence}, both for the strong error and the weak error.  In Section \ref{sec:SGLD}, we apply the modified tamed scheme to the Langevin sampling task and propose the tamed stochastic gradient Langevin dynamics (T-SGLD) algorithm. A rigorous sharp error bound under relative entropy is also obtained in Section \ref{sec:SGLD}. In Section \ref{sec:numerical}, we give some numerical experiments to justify our theoretical findings and show the efficiency of the MTE scheme. Some further discussions are given in Section \ref{sec:discussion} and more technical details during the convergence analysis are given in the Appendix.

         
 \section{The modified tamed scheme}\label{sec:scheme}

In this section, we give more details of the proposed modified tamed scheme. For simplicity, below we will consider the simplest modified tamed Euler (MTE) scheme.

Recall the SDE
	\begin{equation}\label{SDE1}
		d{X}(t)= {b}\big(t, {X}(t)\big)dt+ {{\sigma}}\big(t, {X}(t)\big)d {W}_t,\quad X(0) = X_0,
	\end{equation}    
and the tamed Euler scheme \cite{sabanis2013note}:
\begin{equation*}
d {X_h}(t)= {b^h_{\mathrm{tame}}}\Big(\kappa(t), {X_h}\big(\kappa(t)\big)\Big)dt+ {\sigma}\Big(\kappa(t), {X_h}\big(\kappa(t)\big)\Big)d {W}_t, \quad X_h(0) = X_0,
\end{equation*} 
where
\begin{equation*}
		 {b^h_{\mathrm{tame}}}(t, {x}):=\frac{ {b}(t, {x})}{1+h^{\alpha}| {b}(t, {x})|},
\end{equation*}
for any $t\in [0,T]$ and $\alpha \in (0,1/2]$. Clearly,
\[
b^h_{\mathrm{tame}}-b=\frac{h^\alpha|b|b}{1+h^\alpha|b|}
\]
so the tamed scheme causes an $O(h^{\alpha})$ error, which reduces the convergence rate if the corresponding discretization has convergence order higher than $\alpha$. For the strong error of Euler discretization, this is not an issue because the order is $1/2$ and one may choose $\alpha=1/2$. If we consider weak convergence and higher order schemes, this is clearly a drawback.  Motivated by this, we aim to build a modified taming approach to improve the order of accuracy, compared with the $b^h_{\mathrm{tame}}$ above.

Our observation is that the taming is only needed when $|X|$ is large to prevent numerical instability. Typically $|X|$ is $O(1)$ and the taming error is still present for such values. By this observation, we propose the following modified tamed scheme by introducing a cut-off function in the taming scheme. In particular, we replace the original drift function $b$ in \eqref{SDE} by $b^h$ defined in \eqref{eq:modifiedtamedb} above,
\begin{equation}\label{eq:modifiedtamedb}
		 { b^{h}}(t, {x})=\frac{ {b}(t, {x})}{1+\psi( \gamma h^\alpha| {b}(t, {x})|)}.
	\end{equation}
Here, $\alpha \in (0,\frac{1}{2}]$, $\gamma > 0$, and $\psi: \mathbb{R}_{+} \rightarrow \mathbb{R}_{+}$ is a cut-off function satisfying:
\begin{equation}\label{eq:cutoff}
\psi(r)=\begin{cases}
			\begin{aligned}
			&0, &r\leq 1, \\
			&r, &r\geq 2, \\
			&\text{smooth increasing}, &r\in(1,2).
			\end{aligned}
		\end{cases}
\end{equation}
The parameter $\gamma$ is chosen such that for a major fraction of the values of $X$ in the simulation, $\psi(\gamma h^{\alpha}|b(t, X)|)=0$. 

Compared with the \eqref{eq:originaltamed}, in the modified scheme, 
$b^{h} = b$ if $|b| \leq \gamma^{-1}h^{-\alpha}$, which means that we tame the drift only when $|b|$ is greater than $O(h^{-\alpha})$. On the other hand, 
\begin{equation}\label{eq:bhbound}
| { b^{h}}(t,x)|\leq \min\left(2\gamma^{-1} h^{-\alpha}, | b(t,x)|\right)
\end{equation}
 still holds, so the drift is tamed only when $|b(t,x)|$ is relatively large.  This plays a key role to preventing the moment of the numerical solution $\hat{X}_h(t)$ from growing to infinity. In this sense, our cutoff function $\psi$ shares similarity in spirit with the truncated method. 
The benefit is that we do not have to spare our effort to treat the stopping times as in the truncated schemes.  With the moment bound (see in Lemmas \ref{SE2} and \ref{UITMB} below), the probability of the event $\{|b(t,\hat{X}_h(t))| > h^{-\alpha} \}$ is small due to the Markov inequality. Consequently, we are able to obtain a crucial and improved $L^p$ estimate of the taming error, which can have arbitrary accuracy in terms of the time step $h$, which means that this taming technique can preserve the order of the original explicit discretization for globally Lipschitz coefficients.  See more details in Lemma \ref{SE4}  below.

If we consider the Euler's discretization, the corresponding implementable  modified tamed Euler (MTE) scheme is given by
\begin{equation}
        d {\hat{X}_h}(t)= {b^{h}}\Big(\kappa(t), {\hat{X}_h}\big(\kappa(t)\big)\Big)dt+ {\sigma}\Big(\kappa(t), {\hat{X}_h}\big(\kappa(t)\big)\Big)d {W}_t.
\end{equation}
 Moreover, our scheme allows an inaccurate drift $b^{\xi}$ ($\mathbb{E}_{\xi} b^{\xi} = b$, where $\xi \sim \nu$ is some random variable with distribution $\nu$), which can help largely reduce the computational cost in some practical applications (see for instance \cite{robbins1951stochastic, jin2020random, welling2011bayesian, li2022sharp}). For example, in many practical applications such as the Langevin sampling for Bayesian inference \cite{welling2011bayesian}, one often has
    \begin{equation*}
        b(x)=b_0(x)+\frac{1}{N}\sum_{i=1}^N b_i(x).
    \end{equation*}
In this case, similarly as in SGD, $\xi$ often represents the random minibatch of $\{1,\cdots, N\}$. Namely, for fixed batch-size $S$ ($S$ is a deterministic constant), $\xi$ is the random set $\mathcal{S}=\{(a_1,\dots,a_S): a_i\neq a_j, \forall i\neq j\}$ that is  uniformly randomly  chosen from $\{1,\dots, N \}$. Then the corresponding unbiased estimate $b^{\xi}$ is 
    \begin{equation*}
        b^{\xi}(x)=U_0(x)+\frac{1}{S}\sum_{i=1}^S U_{a_i}(x).
    \end{equation*}   
Then the proposed MTE allowing random batch approximation $\xi$ is given by
\begin{equation}\label{MTE}
		d {\hat{X}_h}(t)= {b^{\xi,h}}\Big(\kappa(t), {\hat{X}_h}\big(\kappa(t)\big)\Big)dt+ {\sigma}\Big(\kappa(t), {\hat{X}_h}\big(\kappa(t)\big)\Big)d {W}_t,
\end{equation} 
where
\begin{equation}\label{eq:modifiedtameddrift}
		 { b^{\xi, h}}(t, {x})=\frac{ {b^\xi}(t, {x})}{1+\psi( \gamma h^\alpha| {b^\xi}(t, {x})|)}.
\end{equation}

It is well-known that the random batch approximation in drift has strong order $1/2$ and weak order $1$ \cite{jin2021convergence}, 
the same as the Euler discretization. Hence, we will show that the MTE scheme with random batch approximation will have order of strong convergence being $1/2$ in this section and order of weak convergence being $1$ in next subsection. The case without random batch approximation is clearly included in these results.

Though in this work, we mainly consider the error analysis of the MTE scheme, we remark that the modified tamed scheme is clearly extensible to other schemes besides the Euler discretization. Consider the Milstein scheme as an example. The Milstein scheme for the SDE \eqref{SDE} is given by \cite{mil1975approximate}:
\begin{multline}\label{eq:milstein}
Y(T_{k+1})=Y(T_k)+b(T_k,Y(T_k))h+\sigma(T_k, Y(T_k)) \cdot (W(T_{k+1}) - W(T_k))\\
+\frac{1}{2} \sigma (T_k,Y(T_k)) \cdot \nabla \sigma (T_k, Y(T_k))(|W(T_{k+1}) - W(T_k)|^2-h),
\end{multline}
where $\nabla$ means the gradient with respect to the space variable. The original tamed Milstein scheme only has a half-order weak convergence order (see for instance \cite[Theorem 5.6]{wang2024weak}).
Accordingly, in our modified tamed Milstein (MTM) scheme, one simply replaces the function $b(t,x)$ in \eqref{eq:milstein} above by \eqref{eq:modifiedtamedb}. Using similar derivations in Sections \ref{sec:strongconvergence} and \ref{sec:weakconvergence}, we can obtain both strong and weak convergence rates for the proposed MTM scheme to be improved to first order. Note that when the noise is additive, i.e., when $\sigma$ does not depend on $x$, the Milstein scheme is reduced to the Euler scheme. Hence, the strong and weak order of MTE with additive noise without random batch would be $1$. On the other hand, if there is random batch approximation, the strong order is then reduced to $1/2$ for MTE, mainly due to the effect of random batch, while the weak order is still $1$.

The following lemmas establish some simple facts for the taming operator. In particular, the modified tamed function $b^h(\cdot)$ shares similar polynomial upper bounds with the original function $b(\cdot)$, as well as the their derivatives, which will be frequently used in our proofs.
\begin{lemma} \label{DPL}
            Assume the function $b:\mathbb{R}^d \rightarrow \mathbb{R}^d$ satisfies $|\nabla b(x)|\leq C(|b(x)|+1)$, then for any cut-off function $\psi:\mathbb{R} \rightarrow \mathbb{R}$ defined in \eqref{eq:cutoff}, the function $b^h=\frac{b}{1+\psi(  h^\alpha|b|)}$ also satisfies $|\nabla b^h(x)|\leq C(|b^h(x)|+1)$.
\end{lemma}
\begin{proof}
By definition of $b^h$, one has
\begin{equation}\label{eq:nablabhexpression}
    \nabla b^h = \frac{\nabla b}{1+\psi(  h^\alpha|b|)}-\frac{ \psi'(  h^\alpha |b|)  h^\alpha b \cdot \nabla b \otimes b  }{(1+\psi(  h^\alpha|b|))^2|b|}.
\end{equation}
Denote $r:=  h^\alpha |b|$. The case $r<1$ is trivial since $b^h = b$.
When $r>2$, $\psi(r)=r$. Then
\begin{equation*}
    \begin{aligned}
                \big| \nabla b^h| = & |\frac{\nabla b}{1+  h^\alpha |b|}-\frac{  h^\alpha b \cdot \nabla b}{(1+  h^\alpha |b|)^2|b|}\otimes b \big|\\
                \leq & \big(1 + \frac{  h^\alpha |b|^2}{|b|(1+  h^\alpha |b|)} \big)   \frac{C(|b|+1)}{1+  h^\alpha |b|} \leq \frac{2C(|b|+1)}{1+  h^\alpha |b|} \leq C' (|b^h|+1).
    \end{aligned}
\end{equation*}
When $1\leq r\leq 2$ (so $  h^\alpha |b| \leq 2$), $\psi(r)\in [1,2]$, $\psi'(r)$ is bounded.  Then
\begin{equation*}
    | \nabla b^h| \leq \frac{C(|b|+1)}{1+\psi(  h^\alpha|b|)} \big(1 + \frac{ \psi'(  h^\alpha |b|)  h^\alpha |b|}{1+\psi(  h^\alpha|b|)} \big) \leq \frac{C'(|b|+1)}{1+\psi(  h^\alpha|b|)} \leq C'' (|b^h|+1).
\end{equation*}
Concluding all the cases above gives the desired result.
\end{proof}

Using a similar proof, one can conclude the following and we omit the detailed proof.
  \begin{lemma}\label{DPL2}
            Assume the function $b:\mathbb{R}^d \rightarrow \mathbb{R}^d$ satisfies $\max\left(|b(x)|, |\nabla b(x)|, |\nabla^2 b(x)|  \right) \leq C(|x|^\ell+1)$, then for any cut-off function $\psi:\mathbb{R}_{+} \rightarrow \mathbb{R}_{+}$ defined in \eqref{eq:cutoff}, there exists $C'$, $\ell'$ such that the function $b^h=\frac{b}{1+\psi(  h^\alpha|b|)}$  satisfies $|\nabla^2 b^h(x)|\leq C'(|x|^{\ell'}+1)$.
   \end{lemma}

Below, in Sections \ref{sec:convergence} - \ref{sec:SGLD}, we conduct error analysis for the proposed MTE scheme. For simplicity, we set $\gamma=1$ in the rest of this paper.

\section{Convergence of the modified tamed Euler scheme}\label{sec:convergence}

In this section, we apply the modified tamed strategy to the basic Euler discretization, which we name modified tamed Euler (MTE) scheme. We perform the convergence analysis in both the strong and weak sense with the possibility that the drift is replaced by its random batch approximation.


\subsection{Strong convergence rate of the MTE scheme}\label{sec:strongconvergence}

The assumptions for the original drift $b$ and the inaccurate drift $b^\xi$ are similar to those in \cite{sabanis2013note}. Note that we only require the one-sided Lipschiz and some polynomial growth condition for the drift, and do not require the global Lipschitz condition which is much more restrictive.

\begin{assumption} \label{SEA} 
Fix any $T>0$. We assume $b^{\xi}$ and $\sigma$ satisfy the following.
   \begin{enumerate}[label=(\alph*)]
 \item There exists a positive constant $K$ such that,
\begin{gather*}
x\cdot {b^\xi} (t, {x}) \vee | {\sigma}(t, {x})|^2 \leq K(1+| {x}|^2)
\end{gather*}
	for any $t\le T$ and $ {x}\in \mathbb{R}^d$.
 \item  There exist  $\ell \geq 1$ and $L>0$ such that, for any $t\le T$, 
 \begin{gather*}
 \begin{split}
&( {x}- {y})^{T}\big( {b^\xi} (t, {x})- {b^\xi} (t, {y})\big)\vee| {\sigma}(t, {x})- {\sigma}(t, {y})|^2\leq L| {x}- {y}|^2,\\
&|b^\xi(t, {x})- b^\xi(t, {y})|\leq L(1+| {x}|^\ell+| {y}|^\ell)| {x}- {y}|
\end{split}
\end{gather*}
	for all $ {x}, {y}\in \mathbb{R}^d$. 

\item The values of $b^{\xi}$ at $(0, 0)$ are bounded $\sup_{\xi}|b^{\xi}(0,0)|<\infty$ and there exist positive constants $M>0$, $m\geq 2$ such that, for any $s,t\in[0,T]$:
	\begin{gather*}
    \begin{aligned}
		| b^\xi(s, {x})- b^\xi(t, {x})|+| \sigma(s, {x})- \sigma(t, {x})| \leq M(1+|x|^m)|s-t|.
	\end{aligned}
    \end{gather*}
        \end{enumerate}
         Above, the constants are all independent of $\xi$.
\end{assumption}


Clearly, the above conditions for $b^{\xi}$ still hold 
for $b$ as $b=\E b^{\xi}$.  Clearly, setting $y=0$ in (b), and $(s, x)=(0, 0)$ in $(c)$, one finds that $b^{\xi}$ and $b$ have a polynomial bound.
One typical example that satisfies Assumption \ref{SEA} is that $\sigma(t,x) \equiv 1$ and $b^{\xi}(t,x) = b(t,x) = -x^3 $ for all $t\geq 0$ and $x \in \mathbb{R}$. Clearly, $b(t,x)$ is not globally Lipschitz in $x$.

\begin{assumption}\label{ass:initial}
For any $p\ge 1$, $\E|X(0)|^p<\infty$.
\end{assumption}

 Theorem \ref{thm:strong} below shows that our MTE scheme has strong convergence for SDEs with such super-linear growing drifts, and the convergence rate is the same as the Euler's scheme. 
 \begin{theorem}\label{thm:strong}
     When Assumptions \ref{SEA} and \ref{ass:initial} hold, the MTE scheme \eqref{MTE} with $\alpha \in (0,1/2]$ converges to the exact solution of SDE \eqref{SDE} with order 1/2, i.e.
     \begin{equation}
         \sup_{0\leq t \leq T} \mathbb{E}[| {X}(t)- {\hat{X}_h}(t)|^2]\leq Ch,
     \end{equation}
	where $C$ is a constant independent of $h$.
 \end{theorem}



Before proving Theorem \ref{thm:strong}, let us give some useful lemmas in the strong error analysis.        
 By Assumptions \ref{SEA} we can derive the moment bounds for the exact solution \eqref{SDE}. The proof is standard and we refer the readers to \cite[Lemma 3.3]{sabanis2013note} or \cite[Lemma 3.1]{wang2024thesis}.
        \begin{lemma}\label{SE1}
            Suppose Assumptions \ref{SEA} and \ref{ass:initial} hold. Fix $p\geq 1$ and $T>0$. Then there exists a positive constant $C$ that depends on $T, K, p, \mathbb{E}(| {X}(0)|^p)$ such that:
		\begin{equation}
			\sup_{0\leq t \leq T} \mathbb{E}| {X}(t)|^p \leq C.
		\end{equation}
        \end{lemma}

 Based on the moment bound above, one can easily obtain the following $L^p$ H\"older's continuity property. Note that they both has half order due to the existence of the Brownian motion. We also omit the proof since this is standard (see for example \cite[Lemma 3.1]{sabanis2013note} ) .
\begin{lemma}\label{SE3'}
		Suppose Assumption \ref{SEA} and \ref{ass:initial} hold. For any $p \geq 2$, there exists a positive constant $C$ that depends on $p, T, K, \mathbb{E}| {X}(0)|^2$ such that
		\begin{equation}
			\sup_{0\leq t\leq T} \mathbb{E}\big| X(t)- X(\kappa(t))\big|^p\leq Ch^{\frac{p}{2}} .
		\end{equation}
\end{lemma}

We can also establish similar results for the numerical solution to \eqref{MTE}. 
        \begin{lemma}\label{SE2}
            Suppose Assumption \ref{SEA} and \ref{ass:initial} hold. Then for any $ p\geq 2$ there exists a positive constant $C$ that depends on $T, K, p, \mathbb{E}(| {X}(0)|^p)$ such that:
		\begin{equation}
			\sup_{h>0} \sup_{0\leq t\leq T} \mathbb{E}(| {\hat{X}_h}(t)|^p)\leq C.
		\end{equation}
        \end{lemma}

        \begin{proof}
    Apply It\^o's formula to $|\hat{X}_h(t)|^p$ and take expectations, it holds that
\begin{multline*}			
\mathbb{E}[| {\hat{X}_h}(t)|^p]
		\le \mathbb{E}[| {X}(0)|^p]+\int_0^t \mathbb{E} \Big[p| {\hat{X}_h}(s)|^{p-2} \hat{X}_h(s) \cdot { b^{\xi, h}}\Big(\kappa(s), {\hat{X}_h}\big(\kappa(s)\big)\Big)\Big]ds\\
        +\int_0^t \mathbb{E}\left[\frac{p(p-1)}{2}\big| {\hat{X}_h}(s)\big|^{p-2} \big| {\sigma}\Big(s, {\hat{X}_h}\big(\kappa(s)\big)\Big)\big|^2\right] ds.
\end{multline*}
Assumption \ref{SEA} and the tamed bound for $ { b^{\xi, h}}$  imply that
\begin{multline*}
\hat{X}_h(s) \cdot { b^{\xi, h}}\Big(\kappa(s), {\hat{X}_h}\big(\kappa(s)\big)\Big)
=\hat{X}_h(\kappa(s)) \cdot { b^{\xi, h}}\Big(\kappa(s), {\hat{X}_h}\big(\kappa(s)\big)\Big)+\\
(\hat{X}_h(s)-\hat{X}_h\big(\kappa(s)\big) )\cdot { b^{\xi, h}}\Big(\kappa(s), {\hat{X}_h}\big(\kappa(s)\big)\Big)
\le K(1+| {\hat{X}_h}(\kappa(s))|^2)+Ch^{-\alpha}|\hat{X}_h(s)-\hat{X}_h\big(\kappa(s)\big)|.
\end{multline*}
Moreover, $| {\sigma}(s, {\hat{X}_h}(\kappa(s)))|^2\le K(1+| {\hat{X}_h}(\kappa(s))|^2)$. Then applying Young's inequality, 
\begin{multline*}
\mathbb{E}[| {\hat{X}_h}(t)|^p]\le  \mathbb{E}[| {X}(0)|^p]+C+C\int_0^t \sup_{0\leq u\leq s }\mathbb{E}(| {\hat{X}_h}(u)|^p)ds\\
			+Ch^{-\frac{\alpha p}{2}}\int_0^t \mathbb{E}| {\hat{X}_h}(s)- {\hat{X}_h}(\kappa(s))|^{\frac{p}{2}}ds.
\end{multline*}

Using the definition of the process $\hat{X}_h$, Jensen's inequality and Burkh\"{o}lder-Davis-Gundy (BDG) inequality \cite{burkholder1971maximal} 
, it holds that
\begin{equation}\label{eq:hatXhcontinuity}
\begin{aligned}
			&\int_0^t \mathbb{E}\Big[\big| {\hat{X}_h}(s)- {\hat{X}_h}(\kappa(s))\big|^{\frac{p}{2}}\Big]ds\\
			\leq & C \int_0^t \mathbb{E}\Big[\big|\int_{\kappa(s)}^s  { b^{\xi, h}}(\kappa(u), {\hat{X}_h}(\kappa(u)))du\big|^{\frac{p}{2}}\Big]+\mathbb{E}\Big[\big(\int_{\kappa(s)}^{s}| {\sigma}(\kappa(u), {\hat{X}_h}(\kappa(u)))|^2du\big)^{\frac{p}{4}}\Big] ds. \\
		\end{aligned}
\end{equation}
For the former, the tamed bound for $b^{\xi,h}$ is used again:
\begin{equation*}
    \int_0^t \mathbb{E}\Big[\big|\int_{\kappa(s)}^s  { b^{\xi, h}}(\kappa(u), {\hat{X}_h}(\kappa(u)))du\big|^{\frac{p}{2}}\Big]ds \leq Ch^{\frac{p(1-\alpha)}{2}}.
\end{equation*}
		For the latter, we use condition (a) in Assumption \ref{SEA} to obtain 
\begin{equation*}
\int_0^t\mathbb{E}\left[\left(\int_{\kappa(s)}^{s}| {\sigma}(\kappa(u), {\hat{X}_h}(\kappa(u)))|^2du\right)^{\frac{p}{4}}\right] ds
\leq Ch^{\frac{p}{4}}\left(1 + \int_0^t \sup_{0\leq u\leq s}\mathbb{E}[| {\hat{X}_h}(u)|^p]ds\right).
\end{equation*}
Combining all the above, one has
\begin{equation*}
\mathbb{E}[| {\hat{X}_h}(t)|^p] \leq C\left(1+\mathbb{E}[| {X}(0)|^2]+\int_0^t\sup_{0\leq u \leq s}\mathbb{E} [| {\hat{X}_h}(u)|^2]ds \right).
\end{equation*}
Clearly, the left hand side can be replaced by $\sup_{0\leq s\leq t}\mathbb{E}[| {\hat{X}_h}(s)|^p]$ since the right hand side is nondecreasing. The claim then follows by Gr\"onwall's inequality.
		\end{proof}

    \begin{lemma}\label{SE3}
		Suppose Assumption \ref{SEA} and \ref{ass:initial} hold. For any $p \geq 2$, there exists a positive constant $C$ that depends on $p, T, K, \mathbb{E}[| {X}(0)|^2]$ such that
		\begin{equation}
			\sup_{0\leq t\leq T} \mathbb{E}[| {\hat{X}_h}(t)- {\hat{X}_h}(\kappa(t))|^p]\leq Ch^{\frac{p}{2}} .
		\end{equation}
        \end{lemma}

        \begin{proof}
		Similarly with \eqref{eq:hatXhcontinuity}, we have
\begin{equation*}
			\mathbb{E}[| {\hat{X}_h}(t)- {\hat{X}_h}(\kappa(t))|^p] 
			\leq  Ch^{p(1-\alpha)}+Ch^{\frac{p}{2}}\left(1+\sup_{\kappa(t)\leq s \leq t}\mathbb{E}| {\hat{X}_h}(\kappa(s))|^p\right).
\end{equation*}
Using the moment bound for $\hat{X}_h$ in Lemma \ref{SE2}, and since $\alpha \in (0,\frac{1}{2}]$, it holds that
\begin{equation*}
\sup_{0\leq t\leq T} \mathbb{E}[| {\hat{X}_h}(t)- {\hat{X}_h}(\kappa(t))|^p]\leq Ch^{\frac{p}{2}} .
\end{equation*}
\end{proof}

The following lemma echoes our motivation, as the drift term of the MTE scheme is near to that of the original in expectation. Moreover, it can be observed from this lemma that the rate of the taming error in our modified scheme can be arbitrarily high in $L^p$ sense.

 \begin{lemma}\label{SE4}
Suppose Assumption \ref{SEA} and \ref{ass:initial} hold. Fix $T>0$. Then for all $t\in [0,T]$, $\alpha \in (0,1/2]$ and $p,q \geq 1$, there exists a positive constant $C$ independent of $h$ such that
		\begin{equation}
			\mathbb{E}\left| {b^\xi}\left(\kappa(t), {\hat{X}_h}(\kappa(t))\right)- { b^{\xi, h}}\left(\kappa(t), {\hat{X}_h}(\kappa(t))\right)\right|^p\leq Ch^{q}.
		\end{equation}
\end{lemma}

        \begin{proof}
		A straight calculation and  H\"{o}lder's inequality gives
\begin{equation*}
\begin{aligned}
			& \mathbb{E}\left| {b^\xi}(\kappa(t), {\hat{X}_h}(\kappa(t)))- { b^{\xi, h}}(\kappa(t), {\hat{X}_h}(\kappa(t)))\right|^p\\ 
			\leq&  h^{p\alpha}\mathbb{E}\left[| {b^\xi}(\kappa(t), {\hat{X}_h}(\kappa(t)))|^{2p}\chi_{\{  h^\alpha | {b^\xi}(\kappa(t), {\hat{X}_h}(\kappa(t)))|\geq 1 \}} \right] \\
			\leq&  h^{p\alpha} \left(\mathbb{E}| {b^\xi}(\kappa(t), {\hat{X}_h}(\kappa(t)))|^{4p}\right)^{\frac{1}{2}} \left(\mathbb{P}(| {b^\xi}(\kappa(t), {\hat{X}_h}(\kappa(t)))|\geq  h^{-\alpha} )\right)^{\frac{1}{2}}.\\
		\end{aligned}
\end{equation*}
Using conditions (b), (c) in Assumption \ref{SEA} as well as the moment bound in Lemma \ref{SE2}, for any $m\geq 1$, one has
$\mathbb{E}\left| {b^\xi}(\kappa(t), {\hat{X}_h}(\kappa(t)))\right|^{m} \leq C$. For the other term, by Markov's inequality, for any $q \geq 1$,
\begin{equation*}
    \mathbb{P}( |{b^\xi}(\kappa(t), {\hat{X}_h}(\kappa(t)))|>h^{-\alpha})\leq \frac{\mathbb{E}|( {b^\xi}(\kappa(t), {\hat{X}_h}(\kappa(t))|^\frac{2q}{\alpha}}{( h^{-\alpha})^\frac{2q}{\alpha}}\leq Ch^{q}.
\end{equation*}
The claim then holds after combining all the above.
\end{proof}
According to this lemma, it is clear that the modified tamed scheme can preserve the order of the original order of the scheme. Now, we are ready to prove the strong convergence.

\begin{proof}[Proof of Theorem \ref{thm:strong}]
Denote $\chi_h(t):=X(t)-\hat{X}_h(t)$, and
\begin{equation*}
    {\beta_h}(t):= {b}(t, {X}(t))- { b^{\xi, h}}(\kappa(t), {\hat{X}_h}(\kappa(t)),\quad	    {\alpha_h}(t):= {\sigma}(t, {X}(t))- {\sigma}(\kappa(t), {\hat{X}_h}(\kappa(t))).
\end{equation*}
By definition, $\chi_h(0)=0$ and $\chi_h$ satisfies for $t\in (0, T]$ that
\begin{equation*}
d {\chi_h}(t)= {\beta_h}(t)dt+ {\alpha_h}(t)d {W}_t .
\end{equation*}
By It\^o's formula, one has
\begin{equation*}
    \mathbb{E}| {\chi_h}(t)|^2=\int_0^t\mathbb{E}[2 {\chi_h}(s) \cdot {\beta_h}(s)]ds+\int_0^t\mathbb{E}| {\alpha_h}(s)|^2ds=: R_1+R_2.
\end{equation*}
   		
The error term $R_2$ is easy to control. In fact, by conditions (a), (c) in Assumption \ref{SEA} and the H\"older's continuity result in Lemma \ref{SE3}, one has:
\begin{equation*}
\begin{split}
  \int_0^t \mathbb{E}[| {\alpha_h}(s)|^2]ds 
   			\leq &C \left(h^{2}+\int_0^t \mathbb{E}[| {X}(s)- {\hat{X}_h}(s)|^2] ds +\int_0^t \mathbb{E}[| {\hat{X}_h}(s)- {\hat{X}_h}(\kappa(s))|^2]ds\right)\\
   			\leq &C \int_0^t\mathbb{E}| {\chi_h}(s)|^2ds+Ch.\\  
\end{split}
\end{equation*}

Next, we estimate the error term $R_1$. 
We split ${\beta_h} (s)$ into following five parts:
\begin{equation}\label{eq:betahsplit}
\begin{aligned}
    \beta_h(s) &= \big( {b}(s, {X}(s))- {b^\xi}(s, {X}(s))\big)
    + \big( {b^\xi}(s, {X}(s))- {b^\xi}(s, {\hat{X}_h}(s))\big)\\
   &\quad + \big( {b^\xi}(s, {\hat{X}_h}(s))- {b^\xi}(\kappa(s), {\hat{X}_h}(s))\big)
     + \big( {b^\xi}(\kappa(s), {\hat{X}_h}(s))- {b^\xi}(\kappa(s), {\hat{X}_h}(\kappa(s)))\big)\\
    &\quad+ \big( {b^\xi}(\kappa(s), {\hat{X}_h}(\kappa(s)))- { b^{\xi, h}}(\kappa(s), {\hat{X}_h}(\kappa(s)))\big) =: \sum_{i=1}^5 \tilde{I}_i.
\end{aligned}
\end{equation}
Define $I_i := 2\int_0^t \mathbb{E}[\chi_h(s) \cdot \tilde{I}_i]ds$, $i=1,2,3,4,5$. We separately estimate these five terms.

Condition (d) of Assumption \ref{SEA} gives the estimate for the term $I_2$:
\begin{equation}\label{eq:I2estimate}
    I_2\leq C\int_0^t \mathbb{E}| {\chi_h}(s)|^2ds.
\end{equation}

By Young's inequality, 
\begin{multline*}
I_3+I_4\leq  \int_0^t \mathbb{E}| {\chi_h}(s)|^2ds +\int_0^t \E| {b^\xi}(s, {\hat{X}_h}(s))- {b^\xi}(\kappa(s), {\hat{X}_h}(s))|^2]ds\\
+\int_0^t \mathbb{E}\Big| {b^\xi}(\kappa(s), {\hat{X}_h}(s))- {b^\xi}(\kappa(s), {\hat{X}_h}\big(\kappa(s))\big)\Big|^2ds.
\end{multline*}
The second term on the right hand side is controlled by $Ch^2$	using Assumption \ref{SEA} (c). The last term on the right hand side is estimated by Assumption \ref{SEA} (b) combined with Lemma \ref{SE2}, Lemma \ref{SE3} and H\"{o}lder's inequality:
\begin{equation*}
\begin{split}
& \int_0^{t} \mathbb{E}\left[(1+| {\hat{X}_h}(s)|^{2\ell}+| {\hat{X}_h}(\kappa(s))|^{2\ell})| {\hat{X}_h}(s)- {\hat{X}_h}(\kappa(s))|^2\right]ds\\
&\leq  C\int_0^{t} (\mathbb{E}[| {\hat{X}_h}(s)- {\hat{X}_h}(\kappa(s))|^4])^{\frac{1}{2}}ds\le  Ch.
\end{split}
\end{equation*}


By Lemma \ref{SE4}, the improved taming error in $I_5$ can be estimated by
\begin{equation*}
I_5\leq \int_0^t \mathbb{E}[| {\chi_h}(s)|^2]ds+Ch^q,
\end{equation*}
for $q\ge 1$ arbitrarily large.

Finally, let us estimate the term $I_1$.
Denote 
\begin{equation*}
    {\varphi_\xi}(t, {X}(t)) := {b}(t, {X}(t))- {b^\xi}(t, {X}(t)).
\end{equation*}
Clearly, since for any $t\geq 0$, $X(t)$ is independent of the random batch $\xi$, 
\begin{equation*}
    \mathbb{E}[ {\varphi_\xi}(t, {X}(t))]= {0}.
\end{equation*}
Moreover, the moment bound of ${\varphi_\xi}(t,X(t))$ can be derived by Lemma \ref{SE1} and condition (c) in Assumption \ref{SEA}. Also,
\begin{equation*}
    \mathbb{E}[ {\chi_h}(\kappa(t)) \cdot {\varphi_\xi}(t, {X}(t))]=0.
\end{equation*}
Then one has:
\begin{equation*}
    \begin{aligned}
   	I_1=&\int_0^t\mathbb{E}[ {\chi_h}(s) \cdot {\varphi_\xi}(s, {X}(s))]ds=\int_0^t\mathbb{E}[( {\chi_h}(s)- {\chi_h}(\kappa(s)))\cdot {\varphi_\xi}(s, {X}(s))]ds\\
   			=& \int_0^t \mathbb{E}\left[(\int_{\kappa(s)}^{s}  {\beta_h}(r)dr)\cdot  {\varphi_\xi}(s, {X}(s))\right]ds+\int_0^t \mathbb{E}\left[(\int_{\kappa(s)}^{s} { \alpha_h}(r)d {W}(r))\cdot {\varphi_\xi}(s, {X}(s))\right]ds\\
      =&I_{1,1} + I_{1,2}.
   		\end{aligned}
\end{equation*}
Using the polynomial bounds of $b$ and $b^{\xi}$ (setting $y=0$ in Assumption \ref{SEA} (b) and the boundedness of $b^{\xi}(t, 0)$ by (c)), one has 
for any $r\le T$ and $s\le T$ that 
\[
\E\left[ | {\beta_h}(r)\cdot  {\varphi_\xi}(s, {X}(s))|\right] \le C(T)
\quad \Longrightarrow \quad I_{1,1} \leq Ch.
\]
Note that $\beta_h \varphi_{\xi}$ is essentially $|b-b^{\xi}|^2$ so the consistency condition $\E b^{\xi}=b$ would not help to improve the rate and the $Ch$ order should be optimal for this term.

For $I_{1,2}$, we note that
\begin{equation*}
    \mathbb{E}\left[\int_{\kappa(s)}^{s}  {\alpha_h}(r)d {W}(r)\cdot \varphi_\xi(\kappa(s), {X}(\kappa(s)))\right]=0.
\end{equation*}
 Combining this with H\"{o}lder's inequality, it holds that
\begin{multline*}
\mathbb{E}\left[\int_{\kappa(s)}^{s}  {\alpha_h}(r)d {W}(r)\cdot {\varphi_\xi}(s, {X}(s))\right]\\ 
\leq  \left(\mathbb{E}\left[|\int_{\kappa(s)}^{s}  {\alpha_h}(r)d {W}(r)|^2\right]\right)^{\frac{1}{2}}\left(\mathbb{E}\Big[\Big| {\varphi_\xi}(s, {X}(s))- {\varphi_\xi}(\kappa(s), {X}(\kappa(s)))\Big|^2\Big]\right)^{\frac{1}{2}}.
\end{multline*}
 By the definition of $ {\alpha_h}$, conditions (a) in Assumption \ref{SEA}, and Lemma \ref{SE1}, \ref{SE2}, we have
\begin{equation*}
\begin{aligned}
 \mathbb{E}\left[|\int_{\kappa(s)}^{s}  {\alpha_h}(r)d {W}(r)|\right]^2 \leq    \mathbb{E}\left[\int_{\kappa(s)}^{s}C(1+| {X}(r)|^2+| {\hat{X}_h}(\kappa(r))|^2)dr\right] \leq  Ch.
\end{aligned}
\end{equation*}
By the definition of $ {\varphi_\xi}$ and similar estimate to $I_4$, one has:
\begin{equation*}
\mathbb{E}[| {\varphi_\xi}(s, {X}(s))- {\varphi_\xi}(\kappa(s), {X}(\kappa(s)))|^2] \leq  C(\mathbb{E}| {X}(s)- {X}(\kappa(s))|^4)^{\frac{1}{2}}+Ch^2 \leq Ch.
\end{equation*}
Hence,  $I_{1,2} \leq Ch$.
Note that the estimate of $I_{1,2}$ is not optimal as one does not make use of the smallness of $\alpha_h$ itself but there is no benefit to do the sharp estimate here.   		
   		
Combining the estimates for $I_1$ - $I_5$, we obtain
\begin{equation*}
        \mathbb{E}| {\chi_h}(t)|^2\leq C\int_0^t\mathbb{E}| {\chi_h}(s)|^2ds+Ch.
\end{equation*}
Then Gr\"onwall's inequality gives the result.
\end{proof}

\subsection{Weak convergence rate of the MTE scheme}\label{sec:weakconvergence}

In this section, we show that unlike existing tamed schemes, our MTE scheme preserve the same weak convergence rate (1st order) as that of the Euler's method. As discussed in Section \ref{sec:intro} and Section \ref{sec:scheme}, the introduction of the cut-off function in the tamed drift largely improves the taming error (see Lemma \ref{SE4} above), thus improving the weak convergence. Our analysis in this section begins with a one-step estimation (see for instance \cite{milstein2004stochastic, wang2024weak}). 

In terms of the weak error, besides the Assumptions \ref{SEA}, we need the following additional polynomial growth assumption for the drift $b$ and its (spatial) derivatives. Moreover, for the diffusion term, we will assume the elliptic condition ($\sigma \sigma^T$ is nondegenerate) as well as the boundedness for the derivatives of $\sigma$. These conditions are common in literature and not restrictive.

\begin{assumption}\label{WEA}
\begin{enumerate} [label=(\alph*)]
 \item  For any $T>0$, there exists positive constants $C$, and $\ell \geq 2$ such that for all $x \in \mathbb{R}^d$ and $t\le T$,
 \begin{equation*}
     |\nabla^2 b(t,x)|+|\nabla^3 b(t,x)|\leq C(1+|x|^{\ell}).
 \end{equation*}
\item For any $T>0$, there exists a positive constant $\lambda$ such that for all $x \in \mathbb{R}^d$ and $t\le T$,  $\Lambda(t,x)\succeq \lambda I$, where $\Lambda := \sigma\sigma^T$.

\item For any $T>0$, there exists positive constant $L$ such that for all $x \in \mathbb{R}^d$ and $t\le T$,
\begin{equation*}
      |\nabla \Lambda (t,x)| +|\nabla^2 \Lambda(t,x)| + |\nabla^3 \Lambda(t,x)| \le L.
\end{equation*}
\end{enumerate}
\end{assumption}
We remark that the extra conditions in Assumption \ref{WEA} are only used to ensure the polynomial growth condition of $v$ defined in \eqref{eq:vdef} below. Hence, these conditions are not imposed on the random approximation $b^{\xi}$. We denote by $C_p^\infty(\mathbb{R}^d)$ all smooth functions whose derivative of any order has a polynomial upper bound.
\begin{theorem}\label{thm:MTEweakconv}
		Suppose Assumptions \ref{SEA}, \ref{ass:initial} and \ref{WEA} hold. Then for any test function $f\in C_p^\infty(\mathbb{R}^d)$, the MTE scheme with $\alpha \in (0,1/2]$ converges weakly to the exact solution of SDE with order $1$. Namely, for small time step $h$ and $X(t)$, $\hat{X}(t)$
        (recall their definitions in Section \ref{sec:scheme} 
        above) sharing the same Brownian motion, there exists a positive constant $C$ independent of the time step $h$ and random batch $\xi$ such that
  \begin{equation}
      \sup_{0\leq t\leq T}|\mathbb{E}[f( {X}(t))]-\mathbb{E}[f( {\hat{X}_h}(t))]|\leq Ch.
  \end{equation}
\end{theorem}

The proof framework is similar to that of Milstein (see in \cite[Section 2.2.1]{milstein2004stochastic} ). Set 
\begin{equation}\label{eq:vdef}
 v(t, {x}, T):=\mathbb{E}[f( {X}(T))| {X}(t)= {x}],
\end{equation}
 where $f\in C_p^\infty(\mathbb{R}^d)$. Notably, we will need the polynomial growth of $v$. 
\begin{lemma}\label{PL}
 Suppose Assumptions \ref{SEA}, \ref{ass:initial} and \ref{WEA} hold. Let $f \in C_p^\infty(\mathbb{R}^d)$, and $v(t, {x}, T)$ defined as above. Then the 0th - 3rd order spatial derivatives have polynomial upper bounds, namely, there exist constants $C_v > 0$ and $\ell_v \geq 1$ such that
 \begin{equation}
     |\nabla^k v| \leq C_v (1 + |x|^{\ell_v}),\quad k=0,1,2,3.
\end{equation}
 \end{lemma}
        The result above under similar assumptions is relatively well-known (see for instance \cite[Lemma 3.4]{wang2024weak}). For completeness, we provide a detailed proof of Lemma \ref{PL} in Appendix \ref{sec:bernstein}, mainly using some PDE techniques including the Bernstein-type estimation \cite{bernstein1906generalisation, bernstein1910generalisation, du2024collision}.


Next, we will get some local moment estimates.  We set $X_{a, {x}}(t)$ to be the solution ${X}(t)$ of \eqref{SDE}  with initial position $x$ and initial time $a$ ($X(a)= {x}$) for $t\ge a$, and $\hat{X}_{a, {x}}(t)$ to be the solution $\hat{X}(t)$ of \eqref{MTE} with initial position $x$ and initial time $a$ ($ \hat{X}(a)= {x}$) for $t\ge a$. We also denote $T_i:=ih$, $T=:Nh$ and $X_i:=X(t_i)$, $\hat{X}_i:=\hat{X}(t_i)$. Consider $X(t)$, $\hat{X}(t)$ sharing the same initial state $\hat{X}_{T_k}$. Usually, one controls $\mathbb{E}|X(h) - \hat{X}(h)|^{2p}$ by $O(h^p)$ due to the presence of Brownian motions. However, making using of the smallness of the difference of $\sigma$, we can actually derive a better $O(h^{2p})$ upper bound.

\begin{lemma}\label{lmm:localmoments}
Suppose Assumptions \ref{SEA} and \ref{ass:initial} hold. Fix $k \in \mathbb{N}$ and denote $X(t)$ and $\hat{X}(t)$ to be $X_{T_k, \hat{X}(T_k)}(t)$ and $\hat{X}_{T_k, \hat{X}(T_k)}(t)$ respectively for $t\ge T_k$, where $\hat{X}_{T_k}$ is the numerical solution at $T_k$. Then, 
\begin{enumerate}[(a)]
\item For any $p \geq 1$,
\begin{equation}\label{eq:pthstability}
    \mathbb{E}|X(T_{k+1}) - \hat{X}(T_{k+1})|^{2p} \leq C h^{2p}.
\end{equation}
\item For any $\mathcal{F}_k$-measurable $\R^d$-valued random variable $Z$,
\begin{equation}\label{eq:nablau}
    \mathbb{E}\left[Z \cdot (X(T_{k+1}) - \hat{X}(T_{k+1})) \right]  \leq C \left( \mathbb{E}|Z|^2\right)^{\tfrac{1}{2}} h^{2},
\end{equation}
\item For any $\mathcal{F}_k$-measurable $\R^{d\times d}$-valued random variable $Z'$,
\begin{equation}\label{eq:nabla2u}
    \mathbb{E}\left[(\hat{X}(T_{k+1}) - X(T_{k})) \cdot Z' \cdot (X(T_{k+1}) - \hat{X}(T_{k+1}))  \right] \leq C\left( \mathbb{E}|Z|^4\right)^{\tfrac{1}{4}}h^2.
\end{equation}
\end{enumerate}
Above, the positive constant $C$ is independent of $h$, $\xi$, $k$.
\end{lemma}

\begin{proof}
(a) By definition, since $X(T_k) = \hat{X}(T_k)$, one has
\begin{multline}\label{eq:hatXhXh}
    X(T_{k+1}) - \hat{X}(T_{k+1}) = \int_{T_k}^{T_{k+1}} \big(b(s,X(s)) - b^{\xi,h}(T_k,\hat{X}(T_k)) \big) ds\\
    + \int_{T_k}^{T_{k+1}} \big(\sigma(s,X(s)) - \sigma(T_k,\hat{X}(T_k)) \big) dW_s.
\end{multline}
For the drift term, note that under the current assumptions, since $\hat{X}(T_k)$ has finite moments, for $s \in [T_k,T_{k+1}]$, it is easy to see that $X(s)$ and $\hat{X}(s)$ also have uniformly bounded finite moments. Then since $b$ and $b^\xi$ have polynomial upper bounds by Assumption \ref{SEA} (also recall that $|b^{\xi,h}| \leq |b|$), it is straightforward to derive the $O(h^{2p})$ upper bound for its $L^{2p}$-norm.  For the diffusion term, BDG inequality implies
\begin{equation*}
    \mathbb{E}\big|\int_{T_k}^{T_{k+1}} \big(\sigma(s,X(s)) - \sigma(T_k,\hat{X}(T_k)) \big) dW_s \big|^{2p} \leq C\mathbb{E}\left(\int_{T_k}^{T_{k+1}}  |\sigma(s,X(s)) - \sigma(T_k,\hat{X}(T_k))|^2 ds \right)^p.
\end{equation*}
Using the conditions for $\sigma$ in Assumptions \ref{SEA} , since $\hat{X}(T_k) = X(T_k)$, for $s \in [0,h]$, it holds
\begin{gather*}
     \mathbb{E}|\sigma(s,X(s)) - \sigma(T_k,\hat{X}(T_k))|^{2p} \leq  Ch^p.
\end{gather*}
Hence, by H\"older's inequality, \eqref{eq:pthstability} holds.

(b) Recall \eqref{eq:hatXhXh} by definition. Clearly, since $Z$ is $\mathcal{F}_k$-measurable, 
\[
\mathbb{E}\left[Z \cdot (X(T_{k+1}) - \hat{X}(T_{k+1})) \right]
=\E\left[Z\cdot\int_{T_k}^{T_{k+1}} \big(b(s,X(s)) - \E_{\xi} b^{\xi,h}(T_k,\hat{X}(T_k)) \big) ds\right].
\]
Making use of the consistency property of the random batch $\xi$ (i.e. $\mathbb{E}b^\xi(x) = b(x)$ for any deterministic $x$ in $\mathbb{R}^d$), we split the above into
\[
\int_{T_k}^{T_{k+1}}\E\left[Z\cdot\big(b(s,X(s)) - b(T_k,\hat{X}(T_k)) \big)
+Z\cdot\E_{\xi}\big( b^{\xi}(T_k,\hat{X}(T_k))- b^{\xi,h}(T_k,\hat{X}(T_k)) \big) \right] ds.
\]
The second term above is the taming error, using Lemma \ref{SE4} again, for any $p,p'>1$, its $p$-th moment has $O(h^{p'})$ upper bound. This is the place where the taming technique plays the key role for weak accuracy improvement.

The first term is the error for standard SDE evolution. It is straightforward to see that
\begin{multline*}
\E[Z\cdot\big(b(s,X(s)) - b(T_k,\hat{X}(T_k)) \big)]\\
\le (\E|Z^2|)^{1/2}\left(\E\left(\E\left[b(s,X(s)) - b(T_k,\hat{X}(T_k))|\mathcal{F}_k\right]\right)^2\right)^{1/2}\le C(\E|Z^2|)^{1/2}(s-T_k).
\end{multline*}
The conclusion then follows.

(c) By It\^o's formula for $Y(t)=(\hat{X}(t) - X(T_k))\cdot Z' \cdot  (X(t) - \hat{X}(T_{k+1})$,
\begin{equation}\label{eq:lmmito}
\begin{aligned}
    &\quad\mathbb{E}\left[(\hat{X}(T_{k+1}) - X(T_k))\cdot Z' \cdot  (X(T_{k+1}) - \hat{X}(T_{k+1}))  \right]\\
    &=\mathbb{E}\left[\int_{T_k}^{T_{k+1}} \big(\hat{X}(s) - X(T_k) \big) \cdot Z' \cdot  (b(s,X(s)) - b^{\xi,h}(T_k,\hat{X}(T_k))) ds
    \right]\\
    &\quad + \mathbb{E}\left[\int_{T_k}^{T_{k+1}}b^{\xi,h}(T_k,\hat{X}(T_k)) \cdot Z' \cdot (X(s) - \hat{X}(s))ds\right] \\
    &\quad + \mathbb{E}\left[\int_{T_k}^{T_{k+1}}\sigma(T_k,X(T_k)) \cdot Z' \cdot \big(\sigma(s,X(s)) - \sigma(T_k,X(T_k)) \big)  ds\right]
\end{aligned}
\end{equation}
For the first term above, the taming error is again very small due to our modified tamed scheme. The problem is that $\E_{\xi}$ cannot be taken directly to $b^{\xi}$ as $\hat{X}$ depends on $\xi$. However, the dependence of $\xi$ only appears in the drift which is not the main term. In particular,
\begin{multline*}
\mathbb{E}\left[\int_{T_k}^{T_{k+1}} \big(\hat{X}(s) - X(T_k) \big) \cdot Z' \cdot  (b(s,X(s)) - b^{\xi}(T_k,\hat{X}(T_k))) ds
    \right]\\
    =\mathbb{E}\left[\int_{T_k}^{T_{k+1}} \E_{\xi}\left( b^{\xi,h}(T_k,\hat{X}(T_k))(s-T_k) \cdot Z' \cdot  (b(s,X(s)) - b^{\xi}(T_k,\hat{X}(T_k))) ds\right)\right]\\
    +\mathbb{E}\left[\int_{T_k}^{T_{k+1}} (\int_{T_k}^s \sigma(T_k,X(T_k)) \cdot dW_{s'})\cdot Z' \cdot  (b(s,X(s)) - \E_{\xi}b^{\xi}(T_k,\hat{X}(T_k))) ds\right].
\end{multline*}
After applying $\E_{\xi}b^{\xi}=b$, we can bound all these as desired. The other two terms in \eqref{eq:lmmito} are similar as in (b), where we need the conditional expectation $\E(\cdot | \mathcal{F}_k)$ to average out the roughness of Brownian motions and the desired bound is easy to obtain. 
\end{proof}

With the lemmas above, we now give the proof of the first-order weak convergence rate. 	
\begin{proof}[Proof of Theorem \ref{thm:MTEweakconv}]
Recall the definitions of $T_k$, $X_{T_k, \hat{X}(T_k)}(T_k+h)$, $\hat{X}_{T_k, \hat{X}(T_k)}(T_k+h)$ above after the statement of Theorem \ref{thm:MTEweakconv}.
Note that $X_0=\bar{X}_0$. Clearly,
\begin{equation}\label{eq:weakproofaux1}
 \mathbb{E}[f( {X}(t_N))]-\mathbb{E}[f( {\hat{X}}_N)]
	= \sum_{i=0}^{N-1}\left(\mathbb{E}[f\left(X_{t_i,\hat{X}_i}(t_N)\right) ]-\mathbb{E}[f\left(X_{t_{i+1},\hat{X}_{t_{i+1}}}(t_N) \right)]\right).
\end{equation}
Consider the function
		$v(s, {x}):=\mathbb{E}[f( {X}_{s, {x}}(t_N))]$, then
\begin{equation*}
    \mathbb{E}[f(X(t_N))]-\mathbb{E}[f(\hat{X}_N)]= \sum_{i=0}^{N-1}(\mathbb{E}[v(t_{i+1},X_{t_i,\hat{X}_i}(t_{i+1})) ] - \mathbb{E}[v(t_{i+1},\hat{X}_{t_i,\hat{X}_i}(t_{i+1})) ]  ).
\end{equation*}
We need to estimate $\mathbb{E}[v(t_{i+1},X_{t_i,\hat{X}_i}(t_{i+1})) ] - \mathbb{E}[v(t_{i+1},\hat{X}_{t_i,\hat{X}_i}(t_{i+1})) ]$ for each $i$.

Fix $i$. For convenience, we use $X(T_i+s)$ and $\hat{X}(T_i+s)$ to replace $X_{T_i, \hat{X}_i}(T_i+s)$ and $\hat{X}_{T_i, \hat{X}_i}(T_i+s)$ for $s \in [0,h)$, but we shall keep in mind they have the same initial data $X(T_i) = \bar{X}(T_i)=\hat{X}_i$ at time $T_i$.  By Taylor's expansion,
\begin{equation}\label{eq:weaktaylor}
    \begin{aligned}
			&\mathbb{E}[v(T_{i+1}, {X}(T_{i+1}))]-\mathbb{E}[v(T_{i+1}, {\hat{X}}(T_{i+1}))]\\
			= & \mathbb{E}[\nabla v(T_{i+1}, {\hat{X}}(T_{i+1}))\cdot ( {X}(T_{i+1})- {\hat{X}}(T_{i+1}))\\
			&+\left( {X}(T_{i+1})- {\hat{X}}(T_{i+1})\right)^{\otimes 2}:\int_0^1  \lambda \nabla^2 v\left(T_{i+1},\lambda X(T_{i+1})+(1-\lambda)\hat{X}(T_{i+1})\right)d\lambda].
		\end{aligned}
\end{equation}

By Lemma \ref{PL}, $\nabla^2 v$ has a polynomial upper bound. Moreover, similarly as in the discussion after \eqref{eq:hatXhXh}, since $\hat{X}(T_i)$ has finite moments by Lemma \ref{SE2}, both $X(T_{i+1})$ and $\hat{X}(T_{i+1})$ have uniformly (in $i$ and $h$) bounded $q$-moments. Then, for the second-order term in \eqref{eq:weaktaylor}, 
one has by H\"older's inequality and Lemma \ref{lmm:localmoments} directly that
\begin{multline*}
\E\left[\left( {X}(T_{i+1})- {\hat{X}}(T_{i+1})\right)^{\otimes 2}:\int_0^1  \lambda \nabla^2 v\left(T_{i+1},\lambda X(T_{i+1})+(1-\lambda)\hat{X}(T_{i+1})\right)d\lambda\right] \\
\le C(\E|{X}(T_{i+1})- {\hat{X}}(T_{i+1})|^4)^{1/2}\le Ch^2.
\end{multline*}

Next, we estimate the first-order term in \eqref{eq:weaktaylor}. By Taylor's expansion for $\nabla v$, 
\begin{equation*}
    \begin{aligned}
			&|\mathbb{E}[\nabla v(T_{i+1}, {\hat{X}}(T_{i+1})) \cdot ( {X}(T_{i+1})- {\hat{X}}(T_{i+1}))]|\\
			\leq & \Big|\mathbb{E}\Big[\nabla v(T_{i+1}, {X}(T_i)) \cdot \Big( {X}(T_{i+1})- {\hat{X}}(T_{i+1})\Big) \Big]\Big|\\
			&+ \Big|\mathbb{E}\Big[\big({\hat{X}}(T_{i+1})- {X}(T_i)\big) \cdot \nabla^2 v(T_{i+1}, {X}(T_i)) \cdot \big( {X}(T_{i+1})- {\hat{X}}(T_{i+1})\big)\Big]\Big|\\
			&+\mathbb{E}\Big[\frac{1}{2}\sup_{\lambda \in [0,1]}| \nabla^3 v(T_{i+1}, Z(\lambda) ))| | { {\hat{X}}}(T_{i+1})- {X}(T_i)|^2| {X}(T_{i+1})- {\hat{X}}(T_{i+1})|\Big],
		\end{aligned}
\end{equation*}
where $Z(\lambda):=\lambda {X}(T_{i+1}) + (1-\lambda){\hat{X}}(T_{i+1})$. Both $\nabla v(T_{i+1}, {X}(T_i))$ and $\nabla^2 v(T_{i+1}, {X}(T_i))$ measurable with respect to $\mathcal{F}_i$.  Moreover, they have finite moments by Lemma \ref{PL} and  the moment bound in Lemma \ref{SE2}. Therefore, by \eqref{eq:nablau} and \eqref{eq:nabla2u} in Lemma \ref{lmm:localmoments}, one has that the first two terms on the right hand side have $O(h^2)$ upper bounds. For the third term, note that $\mathbb{E}|{ {\hat{X}}}(T_{i+1})- {X}(T_i)|^2$ is of $O(h)$. Then one can similarly use the polynomial bound of $\nabla^3 v$ in Lemma \ref{PL}, the moment bound in Lemma \ref{SE2}, \eqref{eq:pthstability} in Lemma \ref{lmm:localmoments}, and H\"older's inequality and derive an $O(h^2)$ upper bound.

Combining the above, we have
\begin{equation*}
    |\mathbb{E}[v(T_{i+1}, {X}_{t, {\bar{X}}(t)}(T_{i+1}))]-\mathbb{E}[v(T_{i+1}, {\bar{X}}_{t, {\bar{X}}(t)}(T_{i+1}))]|\leq Ch^2.
\end{equation*}
Therefore, together with \eqref{eq:weakproofaux1}, the weak error bound follows.
\end{proof}

\begin{remark}
Usually, to prove the weak error of order $p$, one needs the derivatives of the test function to $2p$ for moment matching. Here, we only need the derivatives of $v$ up to third order because we computed the the moments of $X(T_i+h)-\hat{X}(T_i+h)$ instead of the moments of $X(T_i+h)-X(T_i)$ and $\hat{X}(T_i+h)-\hat{X}(T_i)$.
\end{remark}

As we mentioned in the proof of Lemma \ref{lmm:localmoments}, the order of taming has been improved. As a consequence, our MTE is of first order weak accuracy, preserving the original order of accuracy for Euler scheme, even if we choose $\alpha\in (0, 1/2]$.

\section{A near-optimal relative entropy bound for the modified tamed SGLD}\label{sec:SGLD}

As mentioned in Section \ref{sec:intro}, an important application of SDEs is the sampling problem. In this section, we apply the MTE scheme to the Langevin sampling task (in particular, the stochastic gradient Langevin dynamics (SGLD) \cite{welling2011bayesian}), and propose the tamed stochastic gradient Langevin dynamics (T-SGLD). We further rigorously prove that the proposed T-SGLD has a near-second-order accuracy in terms of relative entropy.

Suppose we aim to generate samples from a target distribution $\pi$, which is of the Gibbs form $\pi(x) \propto e^{-\beta U(x)}$, where $\beta > 0$ is the inverse temperature. One classical and effective way to sample from $\pi$ of this kind is using the following overdamped Langevin diffusion, whose invariant measure is exactly $\pi$:
\begin{equation}\label{eq:overdampedlangevin}
    dX(t) = -\nabla U(X(t))dt + \sqrt{2\beta^{-1}}\, dW_t,
\end{equation}
where $W(t)$ is the Brownian motion in $\mathbb{R}^d$. Suppose one can find some suitable numerical scheme to simulate the SDE \eqref{eq:overdampedlangevin}, then after running the scheme in a relatively long time interval, one can then treat the numerical solution at (or after) the terminal time as a sample generated from $\pi$. When the drift term $-\nabla U(\cdot)$ is global Lipschitz (i.e. the Hessian of $\log \pi$ has a global upper bound), the SGLD algorithm proposed in \cite{welling2011bayesian} is an effective approach for such sampling problem. The SGLD algorithm uses a Euler discretization of \eqref{eq:overdampedlangevin}, and uses a random-batch approximation for the drift to reduce the computational cost. It is particularly suited to solve high-dimensional sampling problems such as Bayesian inference, and in the past decades, there exist fruitful results proving the effectiveness of the SGLD algorithm \cite{teh2016consistency, mou2018generalization, li2022sharp}.

Despite the popularity and efficiencyof SGLD, it would still blow up when the Hessian of $\log \pi$ is not bounded. One typical example is $\pi(x) \propto e^{-|x|^p}$ for $p > 2$. This then motivates us to apply our modified tamed scheme to make SGLD effective even in this super-linear growth condition. We name the proposed algorithm the tamed stochastic gradient Langevin dynamics (T-SGLD). In detail, for a constant  time step $h$, the T-SGLD iteration is given by
\begin{equation}\label{eq:tsglddiscrete}
            \bar{X}_{i+1}=\bar{X}_i+hb^{\xi_i,h}(\bar{X}_i)+\sqrt{2\beta^{-1}h} \zeta_i, \quad \zeta_i\sim N(0,I_d), \text{i.i.d}.
\end{equation}
Above, $\xi_0, \dots, \xi_k,\dots$ are random batches i.i.d. picked at each time grid, $b^{\xi_k}(\cdot):= \nabla U^{\xi_k}(\cdot)$, and $U^{\xi_k}(\cdot)$ is an unbised estimate of $U(\cdot)$ satisfying $\mathbb{E}_{\xi_k}U^{\xi_k}(\cdot) = U(\cdot)$. The modified tamed drift $b^{\xi_k,h}$ is then defined in \eqref{eq:modifiedtameddrift} for the given $b^{\xi_k}$. Moreover, in our analysis, we consider a time interpolation of the discrete iteration \eqref{eq:tsglddiscrete}, which is similar to \eqref{MTE}:
\begin{equation}\label{eq:tsgldcontinuous}
            d\bar{X}(t) = b^{\xi_k,h}(\bar{X}(T_k)) dt + \sqrt{2\beta^{-1}} dW,\quad t \in [T_k, T_{k+1}),
\end{equation}
where $T_k := hk$ for $k=0,1,2,\dots$. Note that in order to emphasize that SGLD is a special case of the modified tamed scheme proposed in Section \ref{sec:scheme}, here we use a different notation $\bar{X}$ instead of $\hat{X}_h$.

        We denote $\rho_t$ the time marginal distribution of the Langevin diffusion \eqref{eq:overdampedlangevin}, and $\bar{\rho}_t^{(h)}$ the time marginal distribution of T-SGLD \eqref{eq:tsgldcontinuous}. In what follows, we will give a near-sharp uniform-in-time estimate for the relative entropy $\mathcal{H}(\bar{\rho}_t^{(h)} | \rho_t)$. Note that for two probability measures $\mu$ and $\nu$ on some Polish space $E$, $\mathcal{H}(\mu| \nu)$ is defined by
\begin{gather*}
\mathcal{H}(\mu\| \nu) := \left\{
\begin{aligned}
  &\int_E \log \frac{\mathrm{d} \mu}{\mathrm{d}\nu} \mathrm{d}\mu, & \text{if} ~ \mu \ll \nu,\\
 & \infty, & \text{else},
\end{aligned}\right.
\end{gather*}
where $\frac{\mathrm{d} \mu}{\mathrm{d}\nu}$ is the Radon-Nikodym derivative of $\mu$ with respect to $\nu$.

As a Langevin sampling algorithm, we expect the numerical scheme to have long-time convergence. Therefore, we require the following additional mild conditions (recall that $b = -\nabla U$):
        \begin{assumption} \label{SGLDA}
        \begin{enumerate}[label=(\alph*)]
            \item  There exist constants $C>0$ and $\ell \geq 1$ independent of $\xi$ such that for any $\xi$ and $x$,
        \begin{equation*}
            |b^\xi(x)|\leq C(1+|x|^\ell), \quad |
		\nabla b^\xi(x)| \leq C(|b^\xi(x)|+1), \quad |\nabla^2b^\xi(x)|\leq C(1 + |x|^\ell).
\end{equation*} 
\item The distance between $b^\xi$ and $b$ is uniformly bounded.
\[
\sup_\xi\sup_x |b^\xi(x)-b(x)|<\infty.
  \]
\item There exist constants $R>0, \mu>0, b_0>0$, such that
\begin{equation*}
     \frac{x\cdot b^\xi (x)}{|x||b^\xi (x)|}\leq -\mu, |b^\xi(x)|> b_0, \forall |x|>R,\quad \forall \xi.
\end{equation*}
     
\end{enumerate}
 \end{assumption}
Note that the first condition implies that
$\sup_{\xi} |b^\xi(x)|$ is uniformly bounded in a given compact set.
 Moreover, we need the following assumptions:
\begin{assumption}\label{PGE}
 About the initial data, we assume the following
\begin{enumerate} [label=(\alph*)]
\item The invariant distribution $\pi \propto e^{-\beta U}$ satisfies a Log-Sobolev inequality with a constant $C^{LS}_\pi$, i.e., for all $f$ nonnegative and smooth, the following holds:
\begin{gather*}
\text{Ent}_{\pi}(f):=\int f\log f d\pi-\int f\log\left(\int fd\pi\right) d\pi \leq C^{LS}_{\pi}\int \frac{|\nabla f|^2}{f}d\pi.
\end{gather*}
\item  There exists $\lambda>1$ s.t. the initial distribution $\rho_0$ satisfies:
            $$\frac{1}{\lambda}\leq \frac{\rho_0}{\pi}\leq \lambda$$
and there exists $\delta>0$ such that $\int e^{\delta\sqrt{1+|x|^2}}\rho_0(x)dx<\infty$.
 \end{enumerate}
\end{assumption}

Then, with the above assumptions (note that we still do not have any Lipschitz assumptions for the drift term), we are able to obtain the following near-sharp uniform-in-time relative entropy error bound for the T-SGLD algorithm.

\begin{theorem}\label{thm:tsgld}
        Suppose Assumptions \ref{SGLDA} and \ref{PGE} hold.
		For any $\epsilon \in (0,1)$ , there exists a T-SGLD algorithm with taming parameter $\alpha = \frac{\epsilon}{2}$ such that for small $h$,
        \begin{equation}
            \sup_{t\geq 0}\mathcal{H}(\bar{\rho}_t^{(h)}| \rho_t)\leq Ch^{2-\epsilon},
        \end{equation}
        where $C$ is a positive constant independent of $h$.
\end{theorem}


\paragraph{Overview of the proof} The proof of Theorem \ref{thm:tsgld} seeks some help from \cite{li2022sharp}. The derivation in this section also relies on the Fokker-Planck-type equation associated with the numerical scheme \eqref{eq:tsgldcontinuous}. Note that the tamed-SGLD at discrete time points is a time-homogeneous Markov chain and $\bar{\rho}_{T_k}^{(h)}$  is the law at $T_k$ (recall that $T_k := kh$, $k=0,1,2,\dots$). Let us denote $\rho_{t}^{(\b{\xi},h)}$ the probability density of the fixed-batch version of SGLD for a given sequence of batches $\b{\xi}:=(\xi_0, \xi_1, \dots, \xi_k, \dots)$ so that 
\begin{equation*}
    \bar{\rho}^{(h)}_{t} = \mathbb{E}_{\xi}\left[\bar{\rho}_{t}^{(\b{\xi},h)}\right].
\end{equation*}
Moreover, by the Markov property, we are able to define
\begin{equation}\label{eq:rhoxikh_def}
    \bar{\rho}_t^{(\xi_k,h)} := \mathbb{E}\left[\bar{\rho}_t^{(\b{\xi},h)} \Big| \xi_i,\,i\geq k\right] = \cS_{T_k,t}^{(\xi_k,h)} \bar{\rho}_{T_k}^{(h)}, \quad t \in [T_k,T_{k+1}).
\end{equation}
Above, the operator $\cS_{T_k,t}^{(\xi_k,h)}$ is the evolution operator from $T_k$ to $t$ for the Fokker-Planck equation of the T-SGLD derived in Lemma \ref{FP} below. Namely, for some given $\xi_k$:  
\begin{equation*}
    \partial_t \bar{\rho}_t^{(\xi_k,h)} = -\nabla \cdot (\bar{\rho}_t^{(\xi_k,h)} \hat{b}_t^{\xi_k,h}) + \beta^{-1} \Delta \bar{\rho}_t^{(\xi_k,h)},\quad \bar{\rho}_{T_k}^{(\xi_k,h)}
    =\bar{\rho}_{T_k}^{(h)},
\end{equation*}
where $t \in [T_k,T_{k+1})$ and
\begin{equation*}
    \hat{b}^{\xi_k,h}_t(x) := \mathbb{E}\left[b^{\xi_k,h}\left( \bar{X}_{T_k}\right) |\bar{X}_t = x, \xi_k\right], \quad t \in \left[T_k,T_{k+1}\right).
\end{equation*}
The above Fokker-Planck-type equation enables one to calculate $\frac{d}{dt}\mathcal{H}(\bar{\rho}_t^{(h)} | \rho_t)$. As we will see in \eqref{eq:j1j2j3j4} below, $\frac{d}{dt}\mathcal{H}(\bar{\rho}_t^{(h)} | \rho_t)$ contains four main error terms: (1) the exponential decay term; (2) the time discretization error term; (3) the random batch error term; (4) the taming error term. For (1) - (3), we conduct similar estimations as in \cite{li2022sharp}. The detailed derivations are in fact highly nontrivial and require more careful analysis, since in this paper the drift term may grow super-linearly, which is different from the Lipschitz case considered in \cite{li2022sharp}. For (4), the taming error is handled by Lemma \ref{SE4} above. Recall that due to our modified structure of the tamed drift, the accuracy of the taming error can be arbitrarily high.

In what follows, we will first establish several useful auxiliary lemmas in Section \ref{sec:aux} and postpone the detailed proof of Theorem \ref{thm:tsgld} to Section \ref{eq:mainproofsgld}.
	
\subsection{Auxiliary results}\label{sec:aux}

In order to make our estimate uniform-in-time, we need  a uniform log-Sobolev inequality (LSI) for $\rho_t$ (recall that it is the density of $X_t$ in the Langevin diffusion defined in \eqref{eq:overdampedlangevin}) as follows. 
\begin{proposition}\label{prop:lsi}
Recall the overdamped Langevin equation \eqref{eq:overdampedlangevin} and its time marginal distribution $\rho_t$.
Suppose condition (c) in Assumption \ref{SGLDA} holds. Then, for ant $t \geq 0$, $\rho_t$ satisfies an LSI with a uniform LSI constant $\lambda^2C_{\pi}^{LS}$. 
\end{proposition}

This result is well-known since $q_t := \rho_t / \pi$ evolves according to a backward Kolmogorov equation, which apparently satisfies the maximal principle. The claim then follows by combining the classical Holley-Stroock perturbation lemma \cite{bakry2013analysis}. We refer to \cite[Proposition 3.1]{li2022sharp} for a complete proof and omit the details here.

Different from the previous convergence analysis, in the case of SGLD we need a long-time moment bound. We shall prove that the moment bound for the numerical solution $\bar{X}$ of the modified tamed Euler in infinite time. We consider the Lyapunov function $e^{\delta\sqrt{1+|x|^2} }$ to prove the moment bound. 

    \begin{lemma} \label{UITMB}
        Recall the T-SGLD $\bar{X}(t)$ defined in \eqref{eq:tsgldcontinuous}. Suppose Assumptions \ref{SGLDA} and \ref{PGE} hold. Then there exists $\delta_0 > 0$ such that for any $\delta \in (0,\delta_0)$, there exists a constant $C$ such that
        \begin{equation}
            \sup_{t\geq 0} \mathbb{E}[e^{\delta\sqrt{1+|\bar{X}(t)|^2}}]<C.
        \end{equation}
Consequently, for any $p\geq 1$, there exists $C_p>0$ such that
\begin{equation}
    \sup_{t\geq 0}\mathbb{E}|\bar{X}(t)|^p \leq C_p.
\end{equation}
    \end{lemma}

\begin{proof}
    Recall that the T-SGLD satisfies:
\begin{equation*}
    d\bar{X}(t)=b^{\xi,h}(\bar{X}(\kappa(t)))dt+\sqrt{2}dW_t.
\end{equation*}
We consider the Lyapunov function $f(x)=e^{\delta\sqrt{1+|x|^2}}$, here $\delta$ is a small positive constant to be determined.
It\^o's formula gives:
\begin{equation*}
    \begin{aligned}
        \frac{d}{dt}\mathbb{E}[f(\bar{X}(t))]=&\mathbb{E}\left[f(\bar{X}(t))\frac{\delta \bar{X}(t)}{\sqrt{1+|\bar{X}(t)|^2}}\cdot b^{\xi,h}(\bar{X}(\kappa(t)))\right]\\
        &+ \mathbb{E}\left[f(\bar{X}(t))\left(\frac{\delta^2|\bar{X}(t)|^2}{1+|\bar{X}(t)|^2}+\frac{\delta d}{(1+|\bar{X}(t)|^2)^\frac{1}{2}}-\frac{\delta |\bar{X}(t)|^2}{(1+|\bar{X}(t)|^2)^{\frac{3}{2}}}\right)\right]\\
        =:& E_1+E_2.\\
    \end{aligned}
\end{equation*}
    
For $E_1$, in order to make use of the dissipation assumption ( condition (c)  of Assumption \ref{SGLDA}), we use Ito's formula for $F(\bar{X}) := f(\bar{X})\frac{\delta \bar{X}}{\sqrt{1+|\bar{X}|^2}}$ to obtain that
\begin{equation*}
    \begin{aligned}
        E_1 & 
        \leq  \mathbb{E}\Big[f(\bar{X}(\kappa(t)))\frac{\delta \bar{X}(\kappa(t))}{\sqrt{1+|\bar{X}(\kappa(t))|^2}}\cdot b^{\xi,h}(\bar{X}(\kappa(t)))\\
        &+\int_{\kappa(t)}^t f(\bar{X}(s))(c_1\delta^3+c_2\frac{\delta^2}{\sqrt{1+|\bar{X}(s)|^2}}+c_3\frac{\delta}{1+|\bar{X}(s)|^2})|b^{\xi,h}(\bar{X}(\kappa(t)))| ds\\
        &+\int_{\kappa(t)}^t f(\bar{X}(s))(c_4\delta^2+c_5\frac{\delta}{\sqrt{1+|\bar{X}(s)|^2}})|b^{\xi,h}(\bar{X}(\kappa(t)))|^2ds\Big]\\
        =:&I_1+I_2+I_3.\\
    \end{aligned}
\end{equation*}
The term $I_1$ gives the main source of the dissipation and is essential to make the bound uniform-in-time. Indeed, by condition (c) of Assumption \ref{SGLDA} (also recall the notations $\mu$, $R$ and $b_0$ therein).  
Note that
\begin{gather*}
\frac{\delta \bar{X}(\kappa(t))}{\sqrt{1+|\bar{X}(\kappa(t)}|^2}\cdot \frac{b^\xi(\bar{X}_{\kappa(t)})}{1+\psi(  h^\alpha |b^\xi|)}
\le -\mu \delta \frac{b_0}{1+h^{\alpha}b_0} \chi_{\{|\bar{X}(\kappa(t))|\ge R\}}+C\delta \chi_{\{|\bar{X}|<R\}},
\end{gather*}
where $\sup_{|x|\le R, \xi}|b^{\xi}(x)|<\infty$ by condition (a) and (b) of Assumption \ref{SGLDA}.
Since $b_0/(1+h^{\alpha}b_0)$ is bounded below by $\min(b_0, h^{-\alpha})/2$, one has
\begin{equation*}
 I_1 \leq -c_6 \delta \mathbb{E}[f(\bar{X}(\kappa(t)))] + c_7 \delta .
\end{equation*}

For $I_2$, noting $|b^{\xi, h}|\le h^{-\alpha}$, one has 
\begin{equation*}
\begin{aligned}
    I_2 \leq &  h^{-\alpha}\int_{\kappa(t)}^t \Big(\mathbb{E}[e^{\delta\sqrt{1+|\bar{X}(s)|^2}}(c_8 \delta^2+c_3\delta {R'}^{-2})  \chi_{\{|\bar{X}(s)| \geq R' \}} ]
    + \mathbb{E}[e^{\delta\sqrt{1+{R'}^2}}c_9 \delta \chi_{\{|\bar{X}(s)| < R' \}} ] \Big) ds\\
    \leq & c_{10} h^{1-\alpha}(\delta^2 + \delta {R'}^{-2})\sup_{\kappa(t)\leq s \leq t} \mathbb{E}[e^{\delta\sqrt{1+|\bar{X}(s)|^2}}] + c_{11}h^{1-\alpha} \delta.
\end{aligned}
\end{equation*}
Similarly, 
\begin{equation*}
    I_3 \leq c_{12} h^{1-2\alpha}(\delta^2 + \delta {R'}^{-1})\sup_{\kappa(t)\leq s \leq t} \mathbb{E}[e^{\delta\sqrt{1+|\bar{X}(s)|^2}}] + c_{13}h^{1-2\alpha} \delta.
\end{equation*}
The treatment for $E_2$ is similar (and simpler as there is no $b^{\xi, h}$). 

Combining the above, for $h, \delta < 1$, since $\alpha \in (0,\frac{1}{2}]$, we have
\begin{equation}\label{eq:ddtEfX}
   \begin{aligned}
        \frac{d}{dt}\mathbb{E}[f(\bar{X}_t)]
        \leq  C_1'\delta  -C_2' \delta \mathbb{E}[f(\bar{X}(\kappa(t)))]+\epsilon \delta \sup_{\kappa(t)\leq s \leq t} \mathbb{E}[f(\bar{X}(s))],
    \end{aligned}
\end{equation}
where  $\epsilon := C(\delta + {R'}^{-1})$ can be arbitrarily small for small $\delta$ and large $R'$. Noting $f\ge 1$, one then has
\begin{multline*}
    \mathbb{E}[f(\bar{X}(t))] \leq \mathbb{E}[f(\bar{X}_{\kappa(t)})] + C_1' \delta h +\epsilon \delta h \sup_{\kappa(t)\leq s \leq t} \mathbb{E}[f(\bar{X}(s))]\\
    \le (1+C_1'\delta h)\mathbb{E}[f(\bar{X}_{\kappa(t)})] + \epsilon \delta h \sup_{\kappa(t)\leq s \leq t} \mathbb{E}[f(\bar{X}(s))]. 
\end{multline*}
Consequently, for small $h$,  one has
\begin{equation}\label{eq:localgrowthfx}
    \sup_{\kappa(t)\leq s\leq t} \mathbb{E}[f(\bar{X}(s))]\leq 2\mathbb{E}[f(\bar{X}(\kappa(t)))],\quad \forall t>0.
\end{equation}
Hence, combining \eqref{eq:localgrowthfx} and \eqref{eq:ddtEfX}, it holds that
\begin{equation}\label{eq:EfbarXt}
    \mathbb{E}[f(\bar{X}(t))]\leq \left(1-(C'_2-2\epsilon)\delta h\right)\mathbb{E}[f(\bar{X}(\kappa(t)))]+C'_1\delta h.
\end{equation}
Taking $t=T_k = kh$, $k = 0,1,2,\dots$, choosing small $\delta$ and large $R'$ such that $2\epsilon < C_2'$, one then has
\begin{equation*}
    \sup_k \mathbb{E}[f(\bar{X}(T_k))]<\infty,
\end{equation*}
which along with \eqref{eq:EfbarXt} then implies the desired result.

The moment bound is clearly a direct consequence.
\end{proof}

 To perform the error estimate for the laws, we need to find the evolution of the density of the law. 
 The following equation gives the evolution of the density of the laws.    The proof is well known and is mainly based on Bayes' formula. Its derivations can be found in some related literature, for instance \cite[Lemma 1]{mou2022improved}, \cite[Lemma 2.1]{li2022sharp}, \cite[Lemma 5.1]{li2024error}. We omit the details here.
 \begin{lemma}\label{FP}
 Consider the following time-continuous interpolation of Euler type iteration for $t\in (T_{k-1}, T_k]$ (also recall that $T_k = kh$):
 \begin{equation}
     d \hat{X}_E(t)= b_{k-1}\Big(\hat{X}_E\big(\kappa(t)\big)\Big)dt+ \sigma_{k-1}\Big(\hat{X}_E\big(\kappa(t)\big)\Big)d {W}_t,
 \end{equation}
 and denote its time marginal law by $\hat{\rho}_t$ for $t\geq 0$. Then
        the following Fokker-Planck-type equation holds for $\hat{\rho}_t$: 
       \begin{gather}
       \partial_t \bar{\rho}_t=-\nabla \cdot (\hat{\rho}_t\hat{b}_t)+\frac{1}{2}\nabla^2:( \hat{\Lambda}_t\hat{\rho}_t).
       \end{gather}
If we denote $\Lambda_{k-1}:=\sigma_{k-1}\sigma_{k-1}^T$, then
 \begin{gather}
 \hat{b}_t(x)=\mathbb{E}[ b_{k-1}\left(\hat{X}_E(T_{k-1})\right)|\hat{X}_E(t)=x], \quad 
 \hat{\Lambda}_t(x)=\mathbb{E}[ \Lambda_{k-1}\left(\hat{X}_E(T_{k-1})\right)|\hat{X}_E(t)=x].
\end{gather}
    \end{lemma}

Recall that the definition of $\rho_t^{(\b{\xi},h)}$ defined near \eqref{eq:rhoxikh_def} is the law of the T-SGLD \eqref{eq:tsgldcontinuous} conditioning on the given sequence of batches $\b{\xi}=(\xi_0, \xi_1, \cdots, \xi_k, \cdots)$. In Lemma \ref{lmm:fisher1} below, we first establish 
\begin{equation*}
    \frac{d}{dt} \mathcal{H}(\rho^{(\b{\xi},h)}_t|\pi) \leq -c_0 \mathcal{I}(\rho^{(\b{\xi},h)}_t) + c_1,
\end{equation*}
where we recall $\pi \propto e^{-\beta U}$ and consequently
\begin{equation*}
    \mathcal{H}(\rho^{(\b{\xi},h)}_t|\pi) \leq c,
\end{equation*}
where $c_0$, $c_1$, $c$ are positive constants independent of $t$ and $\b{\xi}$. Then, after similar calculation in of \cite[Lemma 3.2]{li2022sharp}, we are able to prove the following Lemma:  
\begin{lemma}\label{lmm:fisher1}
Suppose Assumption \ref{SGLDA} and \ref{PGE} hold. Recall the definition of $\rho^{(\b{\xi},h)}_t$ near \eqref{eq:rhoxikh_def}. For any $A_0>0$, there exist a positive constant $M$
independent of $\b{\xi}$ such that for small time step $h$, the followings hold:
\begin{equation}
    \sup_{t\geq 0}\int_0^t e^{-A_0(t-s)}\mathcal{I}(\rho^{(\b{\xi},h)}_s)ds
    \leq M.
\end{equation}
\end{lemma}

The detailed derivation is very similar to that of \cite[Lemma 3.2]{li2022sharp}, and we put a sketch in Appendix \ref{sec:fisher1}.

Moreover, during the main proof, we will need the following (time) local growth result for the Fisher information. Remarkably, the proof relies on the fact that the diffusion coefficient is constant, so the density can be written as the convolutional form.

\begin{lemma}\label{lmm:fisher2}
Suppose Assumptions \ref{SGLDA} and \ref{PGE} hold. Fix any $\Delta t \in [0,h]$, and denote $p_{k,\Delta t}(\cdot)$ the density of the random variable $\phi_{\Delta t}(\bar{X}(T_{k-1}))$ with $\phi_{\Delta t}(x) := x + \Delta t b^{\xi,h}(x)$.
Then for small $h$, there exists a positive constant $C$ independent of $h$, $\xi$, $k$, $\Delta t$ such that
\begin{equation}\label{eq:fisherlocal1}
    \int p_{k,\Delta t}(x)|\nabla \log p_{k,\Delta t}(x)|^2 dx \leq 4\mathcal{I}(\rho^{(\b{\xi},h)}_{t_{k-1}}) + C(\Delta t)^{2}.
\end{equation}
Consequently,
\begin{equation}\label{eq:fisherlocal2}
    \mathcal{I}(\rho^{(\b{\xi},h)}_{T_{k-1}+\Delta t}) \leq 4\mathcal{I}(\rho^{(\b{\xi},h)}_{T_{k-1}}) + C(\Delta t)^{2}.
\end{equation}
\end{lemma}

\begin{proof}
We first prove the first claim \eqref{eq:fisherlocal1}. Clearly,
\begin{equation*}
\nabla \phi_{\Delta t}(x) = I + \Delta t \nabla b^{\xi_k,h}(x),\quad \nabla^2 \phi_{\Delta t}(x) = \Delta t \nabla^2 b^{\xi_k,h}(x).
\end{equation*}
By Lemma \ref{DPL}, we know that
\begin{equation*}
    |\nabla b^{\xi,h}(x)|\leq C(b^{\xi_k,h}(x) + 1) \leq C((\Delta t)^{-\alpha} + 1),
\end{equation*}
which implies for small $h$, $\nabla \phi_{\Delta t}$ is reversible and
\begin{equation*}
    |\nabla \phi_{\Delta t}^{-1}| \leq 1 + 2C(\Delta t)^{1-\alpha} .
\end{equation*}
Noting $p_{k,\Delta t}(z) \det(\nabla \phi_{\Delta t}(x))=\rho^{(\b{\xi},h)}_{T_{k-1}}(x)$, one has by change of variable and Young's inequality that
\begin{equation*}
\begin{aligned}
&\quad \int p_{k,\Delta t}(z) |\nabla_x \log p_{k,\Delta t}(z)|^2 dz\\
&=\int \rho^{(\b{\xi},h)}_{T_{k-1}}(x) |\nabla_x\phi_{\Delta t}(x)^{-1} (\nabla_x \log \rho^{(\b{\xi},h)}_{T_{k-1}}(x) - \nabla_x \log \det (\nabla_x \phi_{\Delta t}(x)))|^2 dx\\
    &\leq \int \rho^{(\b{\xi},h)}_{T_{k-1}}(x) |\nabla_x\phi_{\Delta t}(x)^{-1}|^2 \left(2 |\nabla_x \log \rho^{(\b{\xi},h)}_{T_{k-1}}(x)|^2 + 2 |\nabla_x \log \det (\nabla_x \phi_{\Delta t}(x)))|^2\right) dx\\
    &\leq 2(1 + 2C(\Delta t)^{1-\alpha})^2  \mathcal{I}(\rho^{(\b{\xi},h)}_{T_{k-1}}) + 2(1 + C(\Delta t)^{1-\alpha})^4 (\Delta t)^2 \int \rho^{(\b{\xi},h)}_{T_{k-1}}(x)  |\nabla^2_x b^{\xi_k,h}(x)|^2 dx\\
    &\leq 4\mathcal{I}(\rho^{(\b{\xi},h)}_{T_{k-1}}) +  C (\Delta t)^2.
\end{aligned}
\end{equation*}
Above, the value of $C$ may vary from line to line but are all independent of $\xi$, $h$ and $k$, and the third inequality above is due to the polynomial bound for $b^{\xi_k,h}(\cdot)$ and the moment bound in Lemma \ref{UITMB}.

The second claim \eqref{eq:fisherlocal2} is a direct consequence of the first claim \eqref{eq:fisherlocal1} and the Stam's convolution inequality \cite{stam1959some}. In fact, the Stam's convolution inequality states that for any two probability measures $\rho_1$ and $\rho_2$, the inverse of the Fisher information satisfies
\begin{equation*}
    \frac{1}{\mathcal{I}(\rho_1 * \rho_2)} \geq \frac{1}{\mathcal{I}(\rho_1)} + \frac{1}{\mathcal{I}(\rho_2)},
    \end{equation*}
  which implies that $\mathcal{I}(\rho_i)\le \mathcal{I}(\rho_1 * \rho_2)$.
Here the notation $*$ represents the convolution. Since $\rho^{(\b{\xi},h)}_{T_{k-1}+\Delta t} = p_{k,\Delta t} * \mathcal{N}(0,2\beta^{-1}\Delta t)$, the desired bound then follows.
\end{proof}

\subsection{relative entropy error bound for the tamed SGLD algorithm}\label{eq:mainproofsgld}
	With the lemmas above, we can give the relative entropy error bound.

	\begin{proof}[Proof of Theorem \ref{thm:tsgld}]

The main proof idea and structure are similar with the proof of Theorem 3.2 in \cite{li2022sharp}.

		\textbf{Step 1}: Calculate $\frac{d}{dt} \mathcal{H}(\bar{\rho}_t^{(h)}| \rho_t) $.

Recall the definitions of $\bar{\rho}_{t}^{(h)}$, $\rho_{t}^{(\b{\xi},h)}$, $\bar{\rho}_t^{(\xi_k,h)}$ near \eqref{eq:rhoxikh_def} above.
Also recall that due to Lemma \ref{FP},
\begin{equation}\label{FP_xi}
    \partial_t \bar{\rho}_t^{(\xi_k,h)} = -\nabla \cdot (\bar{\rho}_t^{(\xi_k,h)} \hat{b}_t^{\xi_k,h}) + \beta^{-1} \Delta \bar{\rho}_t^{(\xi_k,h)},\quad \bar{\rho}_{T_k}^{(\xi_k,h)}
    =\bar{\rho}_{T_k}^{(h)},
\end{equation}
where $t \in [T_k,T_{k+1})$ and
\begin{equation*}
    \hat{b}^{\xi_k,h}_t(x) := \mathbb{E}\left[b^{\xi_k,h}\left( \bar{X}_{T_k}\right) |\bar{X}_t = x, \xi_k\right], \quad t \in \left[T_k,T_{k+1}\right).
\end{equation*}

Next, we calculate $\frac{d}{dt} \mathcal{H}(\bar{\rho}^{(h)}_t|\rho_t)$ based on \eqref{FP_xi}.
For $t \in [T_k,T_{k+1})$, by \eqref{FP_xi}, we have
\begin{equation}\label{FPbarpi}
\partial_t \bar{\rho}_t^{(h)} = \mathbb{E}_{\xi_k}
\left[-\nabla \cdot (\hat{b}_t^{\xi_k,h} \bar{\rho}_t^{(\xi_k,h)})\right] + \beta^{-1} \Delta \bar{\rho}^{(h)}_t.
\end{equation}
Then direct calculation gives 
\begin{equation*}
    \begin{aligned}
			&\frac{d}{dt}\mathcal{H}(\bar{\rho}_t^{(h)}| \rho_t)  \\
			=&  -\beta^{-1} \int |\nabla \log \frac{\bar{\rho}_t^{(h)}}{\rho_t}|^2\bar{\rho}_t^{(h)}dx + \int \mathbb{E}_{\xi_k}\left[\left(\bar{\rho}_t^{(\xi_k, h)}\hat{b}_t^{\xi_k,h}-\bar{\rho}^{(h)}_tb^{\xi_k}\right)\cdot\nabla \log \frac{\bar{\rho}_t^{(h)}}{\rho_t}\right]dx\\
			& + \int \mathbb{E}_{\xi_k} \left[ (b^{\xi_k}-b)(\bar{\rho}_t^{(\xi_k, h)}-\bar{\rho}_t^{(h)})\nabla \log \frac{\bar{\rho}_t^{(h)}}{\rho_t}\right]dx\\
			\leq & \beta\int \mathbb{E}_{\xi_k}\left[\bar{\rho}_t^{(\xi_k, h)}|\hat{b}_t^{\xi_k,h}-b|^2\right]dx+
            \beta\int \mathbb{E}_{\xi_k\tilde{\xi}_k}\left[|b^{\xi_k}-b|^2\frac{|\bar{\rho}_t^{(\xi_k, h)}-\bar{\rho}_t^{(\tilde{\xi}_k,h)}|^2}{\bar{\rho}_t^{(\tilde{\xi}_k,h)}}\right]dx\\
            &- \frac{1}{2\beta}\int |\nabla \log \frac{\bar{\rho}_t^{(h)}}{\rho_t}|^2\bar{\rho}_t^{(h)}dx. \\
		\end{aligned}
\end{equation*}
Above, we have used Young's inequality in the last inequality. Moreover, $\tilde{\xi}_k$ is another independently chosen random batch at $k$-th time interval, and $\bar{\rho}_t^{(\tilde{\xi}_k,h)}$ is the law corresponding to another realization $\bar{X}'$ of SGLD that shares the same initial ($\bar{X}'_{T_k} = \bar{X}_{T_k}$) and the same Brownian motion in $[T_k, T_{k+1})$ with $\bar{X}$. We further split $|\hat{b}^{\xi_k,h}-b|^2$ via triangular inequality and obtain that
\begin{equation}\label{eq:j1j2j3j4}
    \begin{aligned}
			\frac{d}{dt}\mathcal{H}(\bar{\rho}_t^{(h)}| \rho_t) 
			\leq& 2\beta \mathbb{E}_{\xi_k}\mathbb{E}[|\hat{b}_t^{\xi_k,h}- b^{\xi_k, h}|^2(\bar{X}(t)) |\xi_k]+
            2\beta\mathbb{E}_{\xi_k}\mathbb{E}[| b^{\xi_k, h}-b^{\xi_k}|^2(\bar{X}(t))|\xi_k]\\
            &+			\beta\mathbb{E}_{\xi_k,\tilde{\xi}_k} \int |b^{\xi_k}-b|^2\frac{|\bar{\rho}_t^{(\xi_k, h)}-\bar{\rho}_t^{(\tilde{\xi}_k,h)}|^2}{\bar{\rho}_t^{(\tilde{\xi}_k,h)}}dx
             - \frac{1}{2\beta}\int |\nabla \log \frac{\bar{\rho}_t^{(h)}}{\rho_t}|^2\bar{\rho}_t^{(h)}dx \\
			=&J_1+J_2+J_3+J_4
		\end{aligned}
\end{equation}

Note that $\rho_t$ satisfies a uniform-in-time log-Sobolev inequality due to Proposition \ref{prop:lsi}. Therefore,
\begin{equation*}
    J_4 \leq -\frac{1}{4\beta\lambda^2 C^{LS}_{\pi}}\mathcal{H}(\bar{\rho}_t^{(h)}| \rho_t).
\end{equation*}
In what follows, we aim to bound $J_1$ - $J_3$ by some power of the time step $h$.

		\textbf{Step 2}: Estimate $J_1$. 
		
		We do Taylor expansion for $ b^{\xi, h}$. By definition,
\begin{equation}\label{eq:taylor123}
		\begin{aligned}
			J_1=&2\beta\mathbb{E}_{\xi_k}\mathbb{E}[|\hat{b}_t^{\xi_k,h}- b^{\xi_k, h}|^2(\bar{X}_t) |\xi_k] = 2\beta\mathbb{E}_{\xi_k}\mathbb{E}\left[\left|\mathbb{E}[b^{\xi_k,h}(\bar{X}(T_{k})) - b^{\xi_k, h}(\bar{X}(t))|\bar{X}(t)]\right|^2 \mid \xi_k\right]\\
			 =& 2\beta\mathbb{E}_{\xi_k}\mathbb{E}\Big[\Big|\mathbb{E}\Big[\nabla  b^{\xi, h}(\bar{X}(t))\cdot (\bar{X}(T_{k})-\bar{X}(t)) + r_t
			 \mid \bar{X}(t)\Big]\Big|^2 \mid \xi_k\Big], \\
		\end{aligned}
\end{equation}
where
\begin{equation*}
    r_t := \frac{1}{2}\Big(\int_0^1  \nabla^2 b^{\xi, h}\left(\theta \bar{X}(t)+(1-\theta)\bar{X}(t_{k-1})\right)d\theta\Big) : \left(\bar{X}(t)-\bar{X}(T_{k})\right)^{\otimes2}.
\end{equation*}

    We estimate the first-order term and the remainder term separately.
For the first-order term, we recall that $|\nabla b^{\xi,h}(x)|\leq C(1+| b^{\xi, h}(x)|)\leq Ch^{-\alpha}$. So we have
\begin{multline}\label{eq:keyimprove}
     \mathbb{E}\Big[\Big|\mathbb{E}\Big[\nabla  b^{\xi, h}(\bar{X}(t))\cdot (\bar{X}(T_{k})-\bar{X}(t))\mid \bar{X}(t)\Big]\Big|^2 \mid \xi_k\Big]\\
     \leq Ch^{-2\alpha} \mathbb{E}\left[| \mathbb{E}\left[\bar{X}(t_{k-1})-\bar{X}(t)|\bar{X}(t)\right]|^2 |\xi_k\right].
\end{multline}
By Bayes's law,
\begin{equation*}
    \begin{aligned}
			\mathbb{E}[\bar{X}(t_{k-1})-\bar{X}(t)|\bar{X}(t)=x,\xi_k]=&\int(y-x)P(\bar{X}(T_{k-1})=y|\bar{X}(t)=x,\xi_k)dy \\
			= & \int (y-x)\frac{\bar{\rho}_{T_{k-1}}^{(\xi_k, h)}(y)P(\bar{X}(t)=x|\bar{X}(T_{k-1})=y,\xi_k)}{\bar{\rho}_t^{(\xi_k, h)}(x)}dy.
		\end{aligned}
\end{equation*}
Note that $P(\bar{X}(t)=x|\bar{X}(T_{k-1})=y,\xi_k)$ is a Gaussion, so we split $y-x$ and use integration by parts. 
We claim in Lemma \ref{integrationbyparts} that (see similar details in \cite[Lemma A.3]{li2022sharp})
\begin{equation*}
    \mathbb{E}_{\xi_k}\mathbb{E}\Big[\Big|\mathbb{E}\Big[\nabla  b^{\xi, h}(\bar{X}(t))\cdot (\bar{X}(T_{k})-\bar{X}(t))\mid \bar{X}(t)\Big]\Big|^2 \mid \xi_k\Big]\leq Ch^{2(1-\alpha)}(1+\mathcal{I}(\bar{\rho}_{T_{k-1}}^{(h)})).
\end{equation*}
Compared with \cite[Lemma A.3]{li2022sharp}, the additional $h^{-2\alpha}$ comes from the $O(h^{-\alpha})$ upper bound for $|\nabla b^{\xi_k,h}|$ instead an $O(1)$ one.

For the remainder term in Taylor expansion \eqref{eq:taylor123}, we use H\"{o}lder's inequality, the polynomial upper bound for $\nabla^2 b^{\xi_k,h}$ derived in Lemma \ref{DPL2}, and the moment bound derived in Lemma \ref{UITMB} to obtain that
\begin{equation*}
    \begin{aligned}
			&\quad\mathbb{E}_{\xi_k}\mathbb{E}\Big[\Big|\mathbb{E}\Big[r_t\mid \bar{X}(t)\Big]\Big|^2 \mid \xi_k\Big]
			\leq Ch^{2}.\\
		\end{aligned}
\end{equation*}
Hence,
\begin{equation*}
    J_1 \leq Ch^{2(1-\alpha)}(1+\mathcal{I}(\bar{\rho}_{T_{k-1}}^{(h)})),
\end{equation*}
where $C$ is a positive constant independent of $h$, $k$.

		\textbf{Step 3}: Estimate $J_2$.

$J_2$ is in fact the taming error. Using exactly the same idea of Lemma \ref{SE4}, we are able to derive that
\begin{equation*}
    J_2 \leq Ch^{2}.
\end{equation*}
Note that under Assumption \ref{PGE} (instead of Assumption \ref{SEA} used in Lemmas \ref{SE1}, \ref{SE2} above), the moment bound is uniform-in-time, so the positive constant $C$ here is also independent of $k$. We omit the details here.
	    
 \textbf{Step 4}: Estimate $J_3$. 
 
By Assumption \ref{SGLDA}(b) and the definition of $J_3$ above, we only need to obtain a uniform-in-batch estimate for $\int |\frac{\bar{\rho}_t^{({\xi_k,h})}}{\bar{\rho}_t^{(\tilde{\xi}_k,h)}}-1|^2\bar{\rho}_t^{(\tilde{\xi}_k,h)} dx$.
	    Denote 
        \begin{equation*}
            \tilde{b}(\cdot):=\sqrt{\frac{\beta}{2}}(b^{\tilde{\xi}_k,h}-b^{\xi_k,h})(\cdot).
        \end{equation*}
Since $\bar{\rho}_t^{({\xi_k,h})}$ and $\bar{\rho}_t^{(\tilde{\xi}_k,h)}$ are time marginal laws of two SDEs (recall the definition of $\bar{X}'(t)$ before \eqref{eq:j1j2j3j4} above) that only differs in their drift terms, then,
\begin{equation*}
   \frac{\bar{\rho}_t^{({\xi_k,h})}}{\bar{\rho}_t^{(\tilde{\xi}_k,h)}} (x)
=\mathbb{E} \left[
\exp\left(\int_{T_k}^t \tilde{b} (\bar{X}'(T_k)) dW_s - \frac{1}{2} \int_{T_k}^t |\tilde{b} (\bar{X}'(T_k))|^2 ds\right) \Bigg| \bar{X}'(t) = x, \xi_k, \tilde{\xi}_k  \right].
\end{equation*}
The derivation is based on basic properties of path measures and Girsanov's theorem, we refer to Appendix A.3 of \cite{li2022sharp} for the details. As a result, one has
\begin{multline*}
		\int\Big|\frac{\bar{\rho}_t^{\xi_k}}{\bar{\rho}_t^{\tilde{\xi}_k}}-1\Big|^2\bar{\rho}_t^{\tilde{\xi}_k}dx
		= \int \bar{\rho}_t^{\tilde{\xi}_k}(x) \Big|\int\Big[\exp\big(\sqrt{\frac{\beta}{2}} \tilde{b}(y)(x-y-b^{\tilde{\xi}_k,h}(y)(t-T_{k}))\\
        -\frac{1}{2}|\tilde{b}(y)|^2(t-T_{k})\big)-1 \Big]P(\bar{X}'(T_{k})=y|\bar{X}'(t)=x)dy\Big|^2dx.
\end{multline*}
The treatment here is similar to \textbf{Step 2} as we still split $y-x$ to estimate.  Denote
\begin{equation*}
    z:=\sqrt{\frac{\beta}{2}} \tilde{b}(y)(x-y-b^{\tilde{\xi}_k,h}(y)(t-T_{k})) -\frac{1}{2}|\tilde{b}(y)|^2(t-T_{k}).
\end{equation*}
We split the above into the following three parts:
\begin{equation*}
    \begin{aligned}
		&\exp\left(\sqrt{\frac{1}{2}} \tilde{b}(y)(x-y-b^{\tilde{\xi}_k,h}(y)(t-T_{k})) -\frac{1}{2}|\tilde{b}(y)|^2(t-T_{k})\right)-1\\
		=&\Big(\sqrt{\frac{1}{2}} \tilde{b}(y)(x-y)\Big) + \Big((t-T_{k})(-\sqrt{\frac{1}{2}} b^{\tilde{\xi}_k,h}(y) -\frac{1}{2}|\tilde{b}(y)|^2)\Big) + \Big(e^z-z-1\Big)\\
		=:&K_1+K_2+K_3\\
	\end{aligned}
\end{equation*}
Then it suffices to estimate $\int \bar{\rho}_t^{\tilde{\xi}}(x) |\int K_i P(\bar{Y}_{T_{k-1}}=y|\bar{Y}_t=x)dy|^2dx$, $i=1,2,3$. The estimates for $K_2$ and $K_3$ terms are straightforward, and for the $K_1$ term, we further split it and use integration by parts to estimate its leading term. The details are nearly identical to Lemmas A.4 - A.6 in \cite{li2022sharp}. We give the complete derivation in Appendix \ref{sec:RBerror} for completeness.

Combining the above, we have
\begin{equation*}
    J_3\leq Ch^{2(1-\alpha)}(1+\mathcal{I}(\bar{\rho}^{(h)}_{T_{k}})),
\end{equation*}
where the constant $C$ is independent of $k$, $h$.

    	\textbf{Step 5}: Overall estimate using Gr\"onwall's inequality and the estimate for Fisher information.

            Above, we have derived that for $t \in [T_k,T_{k+1})$,
\begin{equation*}
    \frac{d}{dt}\mathcal{H}(\bar{\rho}_t^{(h)}|\rho_t)\leq C_1h^{2(1-\alpha)}(1+\mathcal{I}(\bar{\rho}^{(h)}_{T_k}))-C_2\mathcal{H}(\bar{\rho}_t^{(h)}|\rho_t),
\end{equation*}
where $C_1$, $C_2$ are independent of $k$ and $h$.
Since the Fisher information $\mathcal{I}(\rho)$ is a convex functional with respect to $\rho$ (see \cite[Lemma 4.2]{fournier2014propagation}), for $t \in [T_k, T_{k+1})$ we have (recall the definition of $\rho^{(\b{\xi},h)}$ near \eqref{eq:rhoxikh_def} above)
\begin{equation*}
\frac{d}{dt}\mathcal{H}(\bar{\rho}_t^{(h)}|\rho_t)\leq C_1h^{2(1-\alpha)}(1+\mathbb{E}_\xi\mathcal{I}(\rho^{(\b{\xi},h)}_{T_k}))-C_2\mathcal{H}(\bar{\rho}_t^{(h)}|\rho_t).
\end{equation*}
By Gr\"onwall's inequality, it holds that
\begin{equation*}
\mathcal{H}(\bar{\rho}_t^{(h)}|\rho_t) \leq C_1h^{2(1-\alpha)}\mathbb{E}_\xi \int_0^t e^{-C_2(t-s)}(1+\mathcal{I}(\rho^{(\b{\xi},h)}_{\kappa(s)}))ds.
\end{equation*}
By Lemma \ref{lmm:fisher2}, we have
\begin{equation*}
    \mathcal{H}(\bar{\rho}_t^{(h)}|\rho_t) \leq C_1' h^{2(1-\alpha)} C_2^{-1}(1-e^{-C_2t}) 
    + C_1'' h^{2(1-\alpha)} \mathbb{E}_{\xi} \int_0^{T_k} e^{-C_2(T_k-s')}\mathcal{I}(\rho_{s'}^{(\b{\xi},h)})  ds'.
\end{equation*}
Note that $C_1'$ is different from $C_1$ due to the constant $\mathcal{I}(\rho^{\b{\xi}}_0)$, and $C_1''$ is changed from $C_1$ by multiplying a constant coming from Lemma \ref{lmm:fisher2} and $e^{C_2 h}$. The small constant $e^{C_1 h}$ appears due to fact that in Lemma \ref{lmm:fisher2}, $\mathcal{I}(\bar{\rho}_{\kappa(s)}^{(\b{\xi},h)})$ is controlled by $\mathcal{I}(\bar{\rho}_{s}^{(\b{\xi},h)})$ for $s \in [\kappa(s)-h, \kappa(s))$ (instead of $[\kappa(s), \kappa(s)+h)$. Finally, by Lemma \ref{lmm:fisher1}, for any $t$ and $\b{\xi}$, there exists a uniform upper bound $M$ such that
\begin{equation*}
    \int_0^{t} e^{-C_2(t-s)}\mathcal{I}(\rho_{s}^{(\b{\xi},h)})  ds \leq M.
\end{equation*}
As a consequence,
\begin{equation*}
    \sup_{t\geq 0}\mathcal{H}(\bar{\rho}_t^{(h)}|\rho_t) \leq Ch^{2-2\alpha}.
\end{equation*}

\end{proof}

\section{Numerical experiment}\label{sec:numerical}

In this section, the MTE scheme is used for numerical experiments to verify the theoretical results. For the strong and weak convergence rate, the 1-dimensional Ginzburg-Landau equation and the 2-dimensional Langevin equation are taken as the examples.

We shall apply the following specific form of the cut-off function $\psi$ in \eqref{eq:cutoff}: 
\begin{gather*}
\psi(r)=\begin{cases}
\begin{aligned}
&0, &r\leq 1, \\
&r, &r\geq 2, \\
&\frac{\exp\left(-(r-1)^{-1}\right)}{\exp\left(-(r-1)^{-1}\right)+\exp\left(-(2-r)^{-1}\right)}\cdot r, &r\in(1,2).
\end{aligned}
\end{cases}
\end{gather*}

\subsection{A 1D SDE with multiplicative noise}

The Ginzburg-Landau equation is used to study phase transitions in superconductivity theory. We consider a 1-dimensional Ginzburg-Landau equation with multiplicative noise: 
            \begin{equation*}
                dX(t)= -(X^3(t)+1.875X(t))dt+0.5X(t)\,dW_t,\quad X(0)=1.
            \end{equation*}
            
 To illustrate the method with random batch approximation, the drift term is split into the super-linear part and the linear part as the idea of mini-batch.

The step-size of the reference solution is set to be $2^{-15}$ while that of the numerical solution is set to be from $2^{-5}$ to $2^{-9}$. The rounds of simulation is $10^5$. The test function is chosen $\cos(x)$ and $\cos(\exp(x))$. The parameters in \eqref{eq:modifiedtamedb} are $\alpha = 0.5$ and $\gamma = 1$.

The reference solution is calculated by modified tamed scheme with reference step-size. The numerical soluitons are given by modified tamed scheme with mini-batch (MTE-RBM), modified tamed scheme without mini-batch (MTE) and tamed scheme with index $\alpha = 1/2$ (TE)[recall \eqref{eq:originaltamed}]. The results are shown in figure \ref{fig:SE1} and \ref{fig:WE1}. Figure \ref{fig:SE1} illustrates the strong error between the reference solution and the numerical solution, verifying the 1/2-order convergence of the MTE scheme. The weak error is given by Figure \ref{fig:WE1}, where the MTE scheme shows a better weak convergence rate of 1-order compared to the 1/2 order weak convergence rate of the original tamed scheme.

\begin{figure}[hbtp]
\centering
\includegraphics[width=0.5\linewidth]{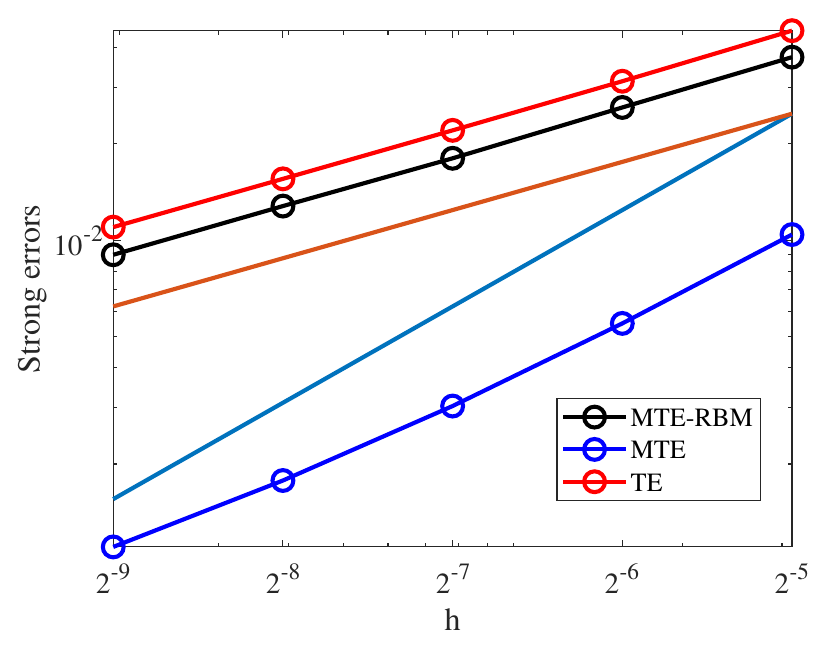}
\caption{Strong convergence for the 1D example with multiplicative noise. All methods are $1/2$ order.}
 \label{fig:SE1}
\end{figure} 

\begin{figure}[hbtp]
\centering
\includegraphics[width=0.48\linewidth]{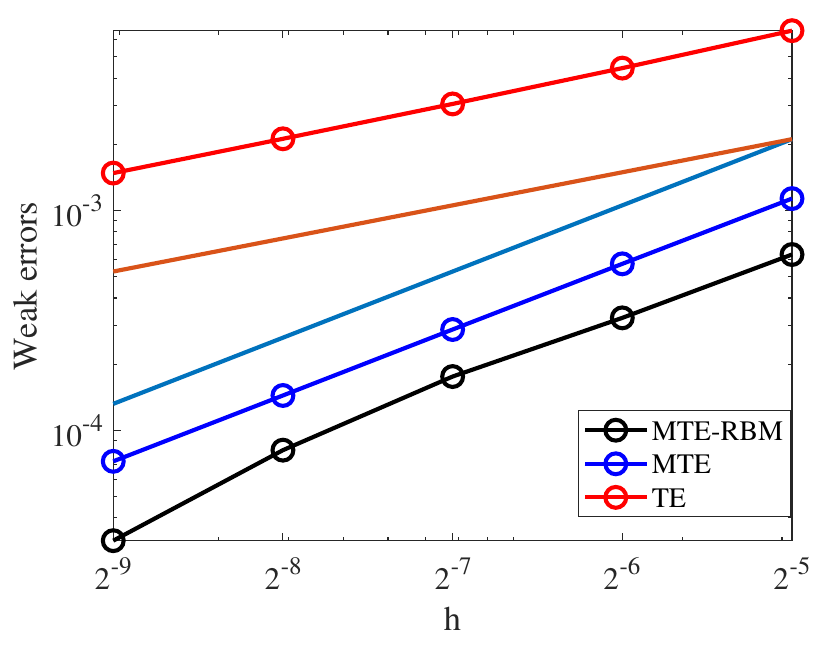}
\includegraphics[width=0.48\linewidth]{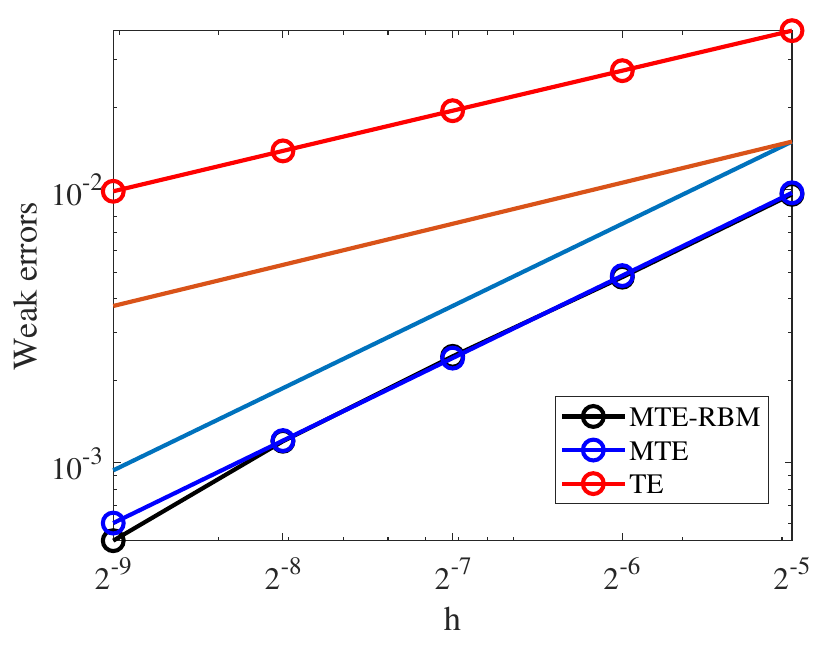}
\caption{Weak convergence for the 1D example. Both the TME with and without random batch are first order, an improvment compared to the original tamed scheme (Test function, Left: $\cos(x)$;  Right: $\cos(\exp(x))$)}
 \label{fig:WE1}
\end{figure} 


\subsection{A 2D Langevin equation with additive noise}
Consider the Langevin diffusion \eqref{eq:overdampedlangevin} mentioned in section \ref{sec:SGLD}, we consider a 2-dimensional case,  where $U(X) = \frac{1}{2}|X|^2-\frac{1}{4}|X|^4$, $\beta=\frac{1}{2}$, $X=[x_1,x_2]^T$:
    \begin{equation*}
        dX(t)=(X(t)-|X(t)|^2\cdot X(t))dt + dW(t), \quad X(0) =[x_1(0),x_2(0)]^T = [1/4, 1/3]^T.
    \end{equation*}
   It is easy to verify that Assumptions \ref{SEA}, \ref{ass:initial} and \ref{WEA} hold. The drift term is split into the super-linear part and the linear part, as the idea of mini-batch.

The step-size of the reference solution is set to be $2^{-17}$ while that of the numerical solution is set to be from $2^{-7}$ to $2^{-12}$. The rounds of simulation is $10^6$. The test function is chosen $\exp(x_1^2+x_2^2)$ and $\cos(\exp(x_1+x_2))$. The parameters in \eqref{eq:modifiedtamedb} are chosen as $\alpha = 0.5$, $\gamma=0.1$.

Schemes of the reference solution and numerical solutions are the same to that of the 1-dimensional case. Simulation results are given in Figure \ref{fig:SE2} and Figure \ref{fig:WE2}. Figure \ref{fig:SE2} tells the 1-order convergence of MTE due to the additive noise, while the MTE-RBM and TE have 1/2-order strong convergence rate due to the random batch error and taming error. Figure \ref{fig:WE2}, similar to Figure \ref{fig:WE1}, verifies that our modified tamed scheme could recover the weak convergence rate to the original scheme.

           \begin{figure}[hbtp]
			\centering
				\centering
				\includegraphics[width=0.5\linewidth]{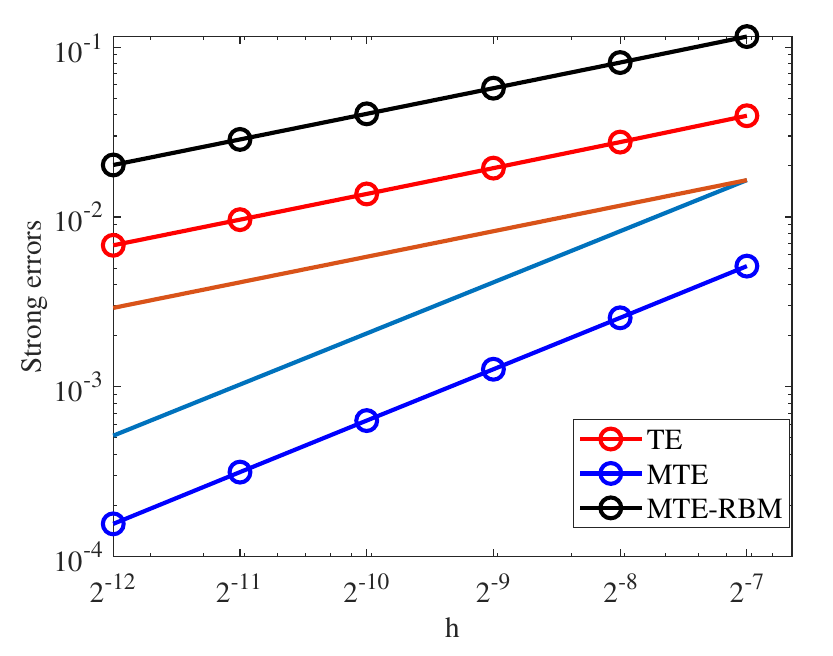}
				\caption{Strong convergence for the 2D Langevin equation with additive noise. The MTE without random batch is of order $1$. MTE with random batch and the original tamed scheme with $\alpha=1/2$ are of order $1/2$.}
                \label{fig:SE2}
		\end{figure} 

            \begin{figure}[hbtp]
			\centering
				\includegraphics[width=0.48\linewidth]{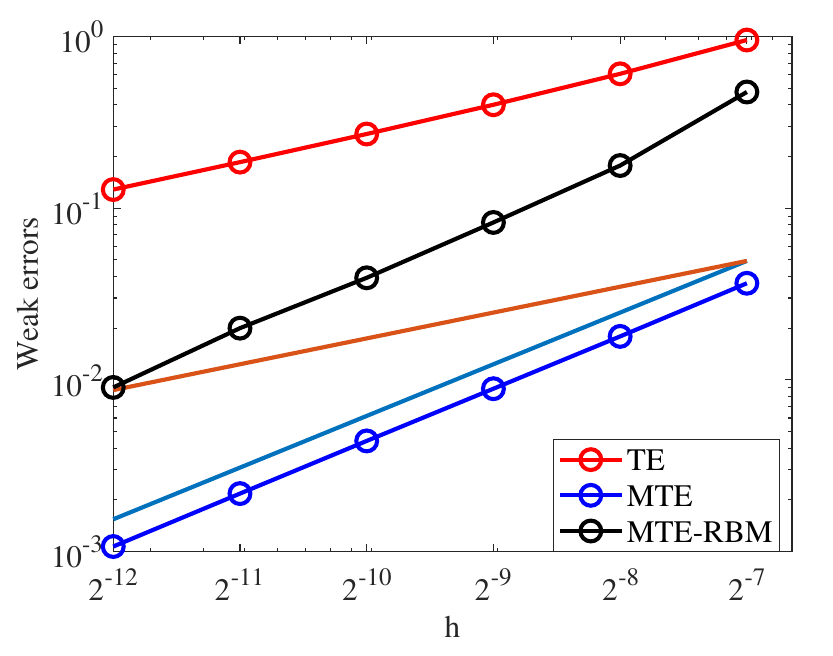}
				\includegraphics[width=0.48\linewidth]{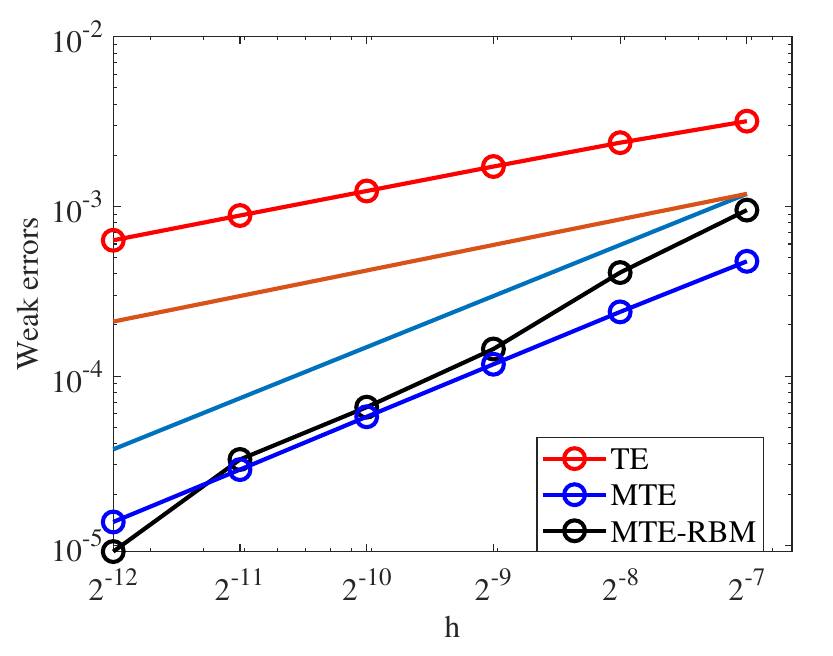}
            \caption{Weak convergence for the 2D Langevin equation. (Test function, Left: $\exp(x_1^2+x_2^2)$;  Right: $\cos(\exp(x_1+x_2))$)}
             \label{fig:WE2}
		\end{figure}

\section{Conclusion and discussion}\label{sec:discussion}

We have proposed a modified tamed scheme for SDEs with nongloabally Lipschitz coefficients by introducing an additional cut-off function in the traditional tamed scheme. This strategy is able to ensure the numerical stability and preserve the original order of convergence for explicit discretization. We have also provided rigorous convergence analysis for the modified tamed Euler (MTE) scheme, and show that the proposed MTE scheme has the desired order of convergence. Moreover, we have also applied the modified taming strategy to Langevin sampling for high-dimensional sampling problems, and obtained a uniform-in-time near-sharp (near second-order) error estimate under KL-divergence.
We would like to emphasize that our modified tamed scheme is not limited to the Euler's case, and can be applied to other higher-order schemes for approximating SDEs with non-globally Lipschitz drifts/diffusions without difficulty. It is also possible to extend to other settings such as stochastic delay differential equations (SDDE), SDEs with time change, and with L\'evy noises. We leave these as future work.

We finally remark here that there still exists room to improve the result in our analysis in Section \ref{sec:SGLD}. Recall that the KL error we obtained in Theorem \ref{thm:tsgld} is only near optimal, which is of order $O(h^{2-\epsilon})$, compared to the sharp results for the Lipschtiz case (see \cite{li2022sharp, mou2022improved}). From technical aspect, in order to obtain a sharp error bound under the current settings, we need to avoid using the $O(h^{-\alpha})$ upper bound for $\nabla b^{\xi,h}$ in places such as \eqref{eq:keyimprove}. Instead, one needs to obtain a uniform-in-time $O(1)$ upper bound for $\int|\nabla \log \bar{\rho}^{(h)}_{t}|^{2+\delta}\bar{\rho}^{(h)}_{t} dx$ for some $\delta > 0$. To our knowledge, this is highly non-trivial. Moreover, a pointwise estimate (which is stronger than the required integral form) for higher-order derivatives of log-(numerical-)densities is in fact very useful in numerical analysis for SDEs, but the derivation is often more tricky and still remains open for the cases of non-Lipschitz drift/diffusion coefficients, degenerate noise, or jumping processes. We leave these as a series of challenging future work.

\section*{Acknowledgement}

The work of L. Li was partially supported by the National Key R\&D Program of China, Project Number 2021YFA1002800,  NSFC 12371400 and 12031013, and  Shanghai Municipal Science and Technology Major Project 2021SHZDZX0102.

\appendix
\section{Proof of Lemma \ref{PL}}\label{sec:bernstein}
Recall that the function $v(t,x)$ satisfies the backward Kolmogorov equation:
\begin{equation*}
    \partial_t v(t,x)+b(t,x)\cdot\nabla v(t,x)   +\frac{1}{2}\Lambda^{ij}(t,x)\partial_{ij}v(t,x)=0, \quad v(T, x)=f(x).
\end{equation*}
Denote $u_T(t,x)=v(T-t,x)$. Then $u_T$ satisfies:
\begin{equation*}
    \partial_t u_T(t,x)=\mathcal{L} u_T(t,x)=b(T-t,x)\cdot\nabla u_T(t,x)  +\frac{1}{2}\Lambda^{ij}(T-t,x)\partial_{ij}u_T(t,x),
\end{equation*}
with initial condition $u_T(0, x)=f(x)$.
Hence it suffices to show that $u$ and its derivatives have polynomial growth.
In what follows, we will simply use $u(t,x)$, $b$, and $\Lambda$ to indicate $u_T(t, x)$, $b(T-x, x)$ and $\Lambda(T-t, x)$ respectively for convenience.		
		
\begin{lemma}
With the Assumptions \ref{SEA}, \ref{ass:initial} and \ref{WEA}, $u$ has polynomial growth, namely, there exists $C_0>0$, $\ell_0 \geq 1$ independent of $t$ and $x$ such that $|u(t,x)| \leq C_0 (1 + |x|^{\ell_0})$.
\end{lemma}

\begin{proof}
We directly make use of the definition of $u(t,x)$ and the assumed polynomial bound for the test function $f(\cdot)$ to obtain that
\begin{equation*}
\begin{aligned}
|u(t,x)| =& \big|\mathbb{E}[f(X(T))|X(T-t)=x]\big|
			\leq  \mathbb{E}[C_f(1+|X(T)|^{\ell_f})|X(T-t)=x]\\
= & \mathbb{E}\left[C_f\left(1+\left|x+\int_{T-t}^T(b(s,X(s))ds+\sigma(s,X(s))dW_s)\right|^{\ell_f}\right)\Big|X(T-t)=x\right].
\end{aligned}
\end{equation*}
Then, by BDG inequality and conditions (a)(b) of Assumption \ref{SEA} (applied to $b$), there exist $C>0$ and $\ell_0 \geq 1$ such that
\begin{equation*}
    |u(t,x)| \leq C+C|x|^{\ell_0}+C\sup_{T-t\leq s\leq T}\mathbb{E}[|X(s)|^{\ell_0} \mid X(T-t)=x].
\end{equation*}
Finally, by the standard moment estimates of SDEs under the given assumptions, 
\[
\sup_{T-t\leq s\leq T}\mathbb{E}[|X(s)|^{\ell_0}\mid  X(T-t)=x] \le C(1+|x|^{\ell_0}).
\]
Consequently,  the conclusion follows.
\end{proof}

For the case of higher order derivatives, our estimations are based on a Bernstein-type method. In what follows, we will state our proof for $\nabla u$ in detail, while omit some similar parts in the proofs for $\nabla^2 u$ and $\nabla^3 u$.


		\begin{lemma}
			With the Assumptions \ref{SEA}, \ref{ass:initial} and \ref{WEA}, $\nabla u$ has polynomial growth, namely, there exist $C_1>0$, $\ell_1 \geq 1$ independent of $t$ and $x$ such that $|\nabla u(t,x)| \leq C_1 (1 + |x|^{\ell_1})$.
		\end{lemma}
	\begin{proof}
		\textbf{Step 1} Construct $w$ as a function of $u$ such that $w$ has a bounded estimation.
		
		We define $w=\frac{|\nabla u|^2}{\sqrt{c-u}}$ and consider a local estimation, i.e. for some given $x_0$,  $x\in B_1( {x_0})$. Here $c$ can be chosen as $(C_02^{l_0}+1)(1+| {x_0}|^{l_0})$, then $c-u\geq 1+| {x_0}|^{l_0}$ for $x\in B_1( {x_0})$.
		
		Consider the operator $ \mathcal{M} = \partial_t - \mathcal{L} $ (recall that $\mathcal{L} = b \cdot \nabla + \frac{1}{2}\Lambda^{ij} \partial_{ij}$). Then by simple calculation we have
		
		$$\begin{aligned}
			\mathcal{M}(w) = & \frac{2u_kb^i_ku_i+u_k\Lambda^{ij}_ku_{ij}-\Lambda^{ij} u_{ki}u_{kj}}{(c-u)^{\frac{1}{2}}}-\frac{\frac{1}{2}u_i\Lambda^{ij}(|\nabla u|^2)_j}{(c-u)^{\frac{3}{2}}}-\frac{\frac{3}{8}\Lambda^{ij}u_iu_j|\nabla u|^2}{(c-u)^{\frac{5}{2}}} \\
			\triangleq & R -\frac{\Lambda^{ij} u_{ki}u_{kj}}{(c-u)^{\frac{1}{2}}}-\frac{\frac{1}{2}u_i\Lambda^{ij}(|\nabla u|^2)_j}{(c-u)^{\frac{3}{2}}}-\frac{\frac{3}{8}\Lambda^{ij}u_iu_j|\nabla u|^2}{(c-u)^{\frac{5}{2}}}. \\
		\end{aligned}$$
		For the remainder term $R$ we use the Young's inequality to obtain that
		 $$\begin{aligned}
			R = & \frac{2u_kb^i_ku_i+u_k\Lambda^{ij}_ku_{ij}}{(c-u)^{\frac{1}{2}}}\leq  2 |\nabla b| \frac{|\nabla u|^2}{\sqrt{c-u}}+\frac{| \Lambda_k |}{\sqrt{c-u}}(\delta u_{ij}^2+\frac{1}{4\delta}|\nabla u|^2)
		\end{aligned}.$$	
		On the other hand, 
		$$\begin{aligned}
		& \frac{\Lambda^{ij} u_{ki}u_{kj}}{(c-u)^{\frac{1}{2}}}+\frac{\frac{1}{2}u_i\Lambda^{ij}(|\nabla u|^2)_j}{(c-u)^{\frac{3}{2}}}+\frac{\frac{3}{8}\Lambda^{ij}u_iu_j|\nabla u|^2}{(c-u)^{\frac{5}{2}}}\\
		= & \Big(\frac{\frac{25}{36}\Lambda^{ij} u_{ki}u_{kj}}{(c-u)^{\frac{1}{2}}}+\frac{\frac{1}{2}u_i\Lambda^{ij}(|\nabla u|^2)_j}{(c-u)^{\frac{3}{2}}}+\frac{\frac{9}{25}\Lambda^{ij}u_iu_j|\nabla u|^2}{(c-u)^{\frac{5}{2}}}\Big)+\frac{\frac{11}{36}\Lambda u_{ki}u_{kj}}{(c-u)^{\frac{1}{2}}}+\frac{\frac{3}{200}\Lambda^{ij}u_iu_j|\nabla u|^2}{(c-u)^{\frac{5}{2}}}\\
		\geq & 0 + \frac{\frac{11}{36}\Lambda^{ij} u_{ki}u_{kj}}{(c-u)^{\frac{1}{2}}}+\frac{\frac{3}{200}\Lambda^{ij}u_iu_j|\nabla u|^2}{(c-u)^{\frac{5}{2}}}.
		\end{aligned}$$
		Choosing $\delta = \frac{11\lambda}{36L_1}$, we have
        \begin{equation*}
            \mathcal{M}(w)\leq -m\frac{w^2}{(c-u)^{\frac{3}{2}}}+Mw\leq -mw^2+Mw.
        \end{equation*}
		Here $m=\frac{3}{200}\lambda$, $M=2|\nabla {b}|+\frac{9L_1}{11\lambda}$.\
        
		\textbf{Step 2} For a given cut-off function $\psi$, compute $\mathcal{M}(\psi w)$.
		
		Now choose a cut-off function $\tilde{\psi}(r)$, such that $\tilde{\psi}(r)$ is strictly positive in $[0,1)$ and supported on $[0,1]$. Moreover, $\tilde{\psi}(r)=1, \forall r\in [0,1/2]$, and $|\tilde{\psi}'|+|\tilde{\psi}''|\leq C_\varepsilon\tilde{\psi}^{\varepsilon}$, $\forall \varepsilon \in (0,1)$.
		
		Set $\psi(x) = \tilde{\psi}(| {x}- {x_0}|)$.		Then we consider $\mathcal{M}(\psi w)=\psi \mathcal{M}(w)+w\mathcal{M}(\psi)-\Lambda^{ij}\psi_iw_j$.
		
		For $\Lambda^{ij}\psi_iw_j$, we have:
		$$\begin{aligned}
			\Lambda^{ij}\psi_iw_j = & \Lambda^{ij}\psi_i\psi^{-1}(\psi w)_j-\Lambda^{ij}\psi_i\psi^{-1}\psi_j w 
		 \geq	 \Lambda^{ij}\psi_i\psi^{-1}(\psi w)_j-|\Lambda| |\tilde{\psi}'|^2\psi^{-1}w.\\
		\end{aligned}$$
		
		For $\mathcal{M}(\psi)$, we have:
		$$\begin{aligned}
			\mathcal{M}(\psi) = & - {b} \cdot \nabla \psi - \frac{1}{2}\Lambda^{ij}\partial_{ij}\psi
			\leq  d| {b} | \cdot |\tilde{\psi}'|+\frac{d^2}{2} |\Lambda|\cdot|\tilde{\psi}''|.\\
 		\end{aligned}$$
		
		Recall that $\mathcal{M}(w)\leq -mw^2+Mw$, and we choose $\varepsilon = \frac{1}{2}$ (We may assume $C_\frac{1}{2}\geq 1$ without loss of generality), then we have:
		
		$$\begin{aligned}
			\mathcal{M}(\psi w) \leq & M\psi w-m \psi w^2 +(d| {b}|+\frac{d^2}{2}|\Lambda|)C_{\frac{1}{2}}\tilde{\psi}^{\frac{1}{2}} w - \Lambda^{ij}\psi_i\psi^{-1}(\psi  w)_j + |\Lambda| C_{\frac{1}{2}}^2 w^2.\\
		\end{aligned}$$
		
		Then we have:
		$$\mathcal{M}(\psi w) + \Lambda^{ij}\psi_j\psi^{-1}(\psi w)_j\leq (M+d| {b}|+\frac{d^2}{2}\lambda+\lambda)C_{\frac{1}{2}}^2 w-m\psi w^2\triangleq \tilde{M} w-m\psi w^2.$$\\
		\textbf{Step 3} Consider the maximum points.
		
		Notice that the left-hand side is a parabolic equation for $\psi w$, here $t\in[0,T]$,  ${x}\in B_1( {x_0}) $, we have:
		
		If $\psi w$ attains its maximum at $(0, {x})$, then:
		
		$$ w(t, {x_0})\leq \psi(x) w(0, {x})\leq \frac{|\nabla u(0, {x})|^2}{\sqrt{c-u(0,x)}}\leq |\nabla f(x)|^2.$$
		
		This leads to:
		\[
        |\nabla u|^2(t, {x_0})\leq |\nabla f(x)|^2\sqrt{c-u}\leq 8(C_T2^{\ell_f})^2(1+| {x_0}|^{3\ell_f}). 
        \]
		
		Hence we have: $$|\nabla u|\leq 4C_T2^{\ell_f}(1+| {x_0}|^{3\ell_f}).$$
		
		If $\psi w$ attains its maximum at $(t_0, {x})$ with $t_0>0$, then 
		$$\begin{aligned}
			0\leq & \mathcal{M}(\psi w) + \Lambda^{ij}\psi_j\psi^{-1}(\psi w) 
			\leq & \tilde{M} w-m\frac{\psi w^2}{(c-u)^{\frac{3}{2}}}.\\
		\end{aligned}$$ 
		So $$ w(t, {x_0})\leq \psi( {x}) w(t, {x})\leq \frac{\tilde{M}}{m}(c-u)^{\frac{3}{2}}=\frac{200}{3\lambda}(2|\nabla  {b}|+\frac{9L_1}{11\lambda}+d| {b}|+\frac{d^2}{2}\lambda+\lambda)(c-u)^{\frac{3}{2}}\leq \tilde{C}(1+| {x_0}|^{\tilde{\ell}}).$$
		Similarly, this leads to a polynomial growth estimation for $u_i$.
		
		Concluding the above, we know that $\nabla u$ has a polynomial bound. 
		\end{proof}
		
		\begin{lemma}
			With the Assumptions \ref{SEA}, \ref{ass:initial} and \ref{WEA}, $\nabla^2 u$ has polynomial growth, namely, there exist $C_2>0$, $\ell_2 \geq 1$ independent of $t$ and $x$ such that $|\nabla^2 u(t,x)| \leq C_2 (1 + |x|^{\ell_2})$.
		\end{lemma}
		\begin{proof} \textbf{Step 1} Construct $ w$ based on $u$ such that $ w$ has a bounded estimation. 
			
		We consider $ w=\frac{\sum_{k,\ell}u_{k\ell}^2}{\sqrt{c-|\nabla u|^2}}$ and still use a local estimation. Here $c$ may be chosen as $(C_1^22^{2\ell_1+1}+1)(1+| {x_0}|^{2\ell_1})$. Then $c-|\nabla u|^2\geq 1+| {x_0}|^{2\ell_1}$ for $x\in B_1( {x_0})$.
		
		By simple calculation, we have:
		$$\begin{aligned}
			\mathcal{M}( w)=&\frac{2\sum_{k,\ell}u_{k\ell}(b_k\cdot \nabla u_\ell+b_\ell\cdot \nabla u_k+b_{k\ell}\cdot \nabla u+\frac{1}{2}\Lambda^{ij}_{k\ell}u_{ij}+\frac{1}{2}\Lambda^{ij}_ku_{ij\ell}+\frac{1}{2}\Lambda^{ij}_lu_{ijk})}{\sqrt{c-|\nabla u|^2}}\\
			&+\frac{\sum_{k}u_k(b_k\nabla u+\frac{1}{2}\Lambda_k^{ij}u_{ij})\sum_{k,\ell}u_{k\ell}^2}{(\sqrt{c-|\nabla u|^2})^3}
			-\Lambda^{ij}\frac{\sum_{k,\ell}u_{k\ell i}u_{k\ell j}}{\sqrt{c-|\nabla u|^2}}\\
			&-\Lambda^{ij}\frac{2\sum_k u_ku_{ki}\sum_{k,\ell}u_{k\ell}u_{k\ell j}+\frac{1}{2}\sum_ku_{ki}u_{kj}\sum_{k,\ell}u_{k\ell}^2}{(\sqrt{c-|\nabla u|^2})^3}\\
			&-\frac{3}{2}\Lambda^{ij}\frac{\sum_ku_ku_{ki}\sum_ku_ku_{kj}\sum_{k,\ell}u_{k\ell}^2}{(\sqrt{c-|\nabla u|^2})^5}.\\
		\end{aligned}$$
		
		Using the assumptions above,  one can similarly derive that
	\[
    \begin{aligned}
			\mathcal{M}( w)\leq & M w-\tilde{M} w^2+\frac{C_b}{2^{\ell_1}}(1+|x_0|^{\ell_b}) \\
		\end{aligned}
    \]
		Here $M=4dCC_1^2C_b(1+|x_0|^{2\ell_1+\ell_b})+\frac{dL}{2\eta}CC_1^2(1+|x_0|^{2\ell_1})$, $\tilde{M}=\frac{\lambda}{2\sqrt{c}}+ \frac{d\lambda}{3C(1+|x_0|^{\ell_1})}-\frac{dL\eta}{4}.$
		Also, if one chooses $\eta = \frac{\lambda}{CL(1+|x_0|^{\ell_1})}$, then $M>0$ and $\tilde{M}>0$.
		
		The remaining two steps are very similar to the 1st-order case. The only difference is the case where $\psi  w$ attains its maximum at $(t_0,x)$ with $t_0>0$.
		
		In detail, after the same calculations, we get ${N} w-\tilde{N}\psi  w^2+\frac{C_b}{2^{\ell_1}}(1+|x_0|^{\ell_b}) \geq 0$	with $N>0$ and $\tilde{N}>0$. Multiplying by $\psi$ on both sides, we have:
		$$\tilde{N}(\psi  w)^2-N\psi w-\frac{C_b}{2^{\ell_1}}(1+|x_0|^{\ell_b})\leq 0$$	
		
		Solve the equation we can get the polynomial growth in this case. 
		Thus we can derive that $|u_{k\ell}|$ has a polynomial upper bound.
		\end{proof}

		\begin{lemma}
			With the Assumptions \ref{SEA}, \ref{ass:initial} and \ref{WEA}, $\nabla^3 u$ has polynomial growth, namely, there exist $C_3>0$, $\ell_3 \geq 1$ independent of $t$ and $x$ such that $|\nabla^3 u(t,x)| \leq C_3 (1 + |x|^{\ell_3})$.
		\end{lemma}

            The 3-order case is very similar to the 2-order case, as we still consider $ w=\frac{\sum_{k,\ell,m}u_{k\ell m}^2}{\sqrt{c-\sum_{k,\ell}u_{k\ell}^2}}$ and give a local estimation. Here $c$ may be chosen as $(C_2^22^{2\ell_2+1}+1)(1+| {x_0}|^{2\ell_2})$(for a fixed $ {x_0}$). Then $c-\sum_{k,\ell}u_{k\ell}^2\geq 1+| {x_0}|^{2\ell_2}$ for $x\in B_1( {x_0})$. 
            
            The simple calculation of the auxiliary function can verify it:
            
		$$\begin{aligned}
			\mathcal{M}( w) = & \frac{2\sum_{k,\ell,m}u_{k\ell m}\left[\left( b\cdot \nabla u +\frac{1}{2}\Lambda^{ij}u_{ij}\right)_{k\ell m}-b\cdot \nabla u_{k\ell m}-\frac{1}{2}\Lambda^{ij}u_{k\ell mij}\right]}{\sqrt{c-\sum_{k,\ell}u_{k\ell}^2}}\\
			+& \frac{\sum_{k,\ell,m}u_{k\ell m}^2\sum_{k,\ell}u_{k\ell}[\left(b\cdot \nabla u+\frac{1}{2}\Lambda^{ij}u_{ij}\right)_{k\ell}-b\cdot \nabla u_{k\ell}-\frac{1}{2}\Lambda^{ij}u_{k\ell}u_{k\ell ij}]}{(\sqrt{c-\sum_{k,\ell}u_{k\ell}^2})^3}\\
			-& \Lambda^{ij}\frac{u_{k\ell mi}u_{k\ell mj}}{\sqrt{c-\sum_{k,\ell}u_{k\ell}^2}}- \Lambda^{ij}\frac{2\sum_{k,\ell,m}u_{k\ell m}u_{k\ell mi}\sum_{k,\ell}u_{k\ell}u_{k\ell j}+\frac{1}{2}\sum_{k,\ell}u_{k\ell i}u_{k\ell j}\sum_{k,\ell,m}u_{k\ell m}^2}{(\sqrt{c-\sum_{k,\ell}u_{k\ell}^2})^3}\\
			-& \frac{3}{2}\Lambda_{ij}\frac{\sum_{k,\ell}u_{k\ell}u_{k\ell i}\sum_{k,\ell}u_{k\ell}u_{k\ell j}\sum_{k,\ell,m}u_{k\ell m}^2}{(\sqrt{c-\sum_{k,\ell}u_{k\ell}^2})^5}\\
		\end{aligned}$$
	
		The trick is the same as the 2nd-order case, and we can also derive $\mathcal{M}( w)\leq M w-\tilde{M} w^2+C$ with positive $M, \tilde{M}, C$. Thus, we have that $|u_{k\ell m}|$ has a polynomial upper bound.

\section{Missing proofs for the relative entropy bound estimate}

\subsection{Proof of Lemma \ref{lmm:fisher1}}\label{sec:fisher1}

\begin{proof}

We first prove the following claims:

\textbf{Claim 1}. There exist positive constants $c_0$, $c_1 $ independent of $t$, $h$, $\b{\xi}$ such that
\begin{equation*}
    \frac{d}{dt} \mathcal{H}(\bar{\rho}^{(\b{\xi},h)}_t|\pi) \leq -c_0  \mathcal{I}(\bar{\rho}^{(\b{\xi},h)}_t) + c_1.
\end{equation*}

\textbf{Claim 2}. There is a positive constant $c_2 $ independent of $t$, $h$, $\b{\xi}$ ($c_2$ depends on $\mathcal{H}(\rho_0|\pi)$) such that
\begin{equation*}
    \mathcal{H}(\bar{\rho}^{(\b{\xi},h)}_t|\pi) \leq c_2.
\end{equation*}

Indeed, for Claim 1, using Fokker-Planck equation \eqref{FP} for $\bar{\rho}^{(\b{\xi},h)}_t$, we have
\begin{equation*}
    \begin{aligned}
    \frac{d}{dt}\mathcal{H}(\bar{\rho}^{(\b{\xi},h)}_t|\pi) = \int \left(\nabla \log \bar{\rho}^{(\b{\xi},h)}_t - \beta b\right)\cdot \left(\bar{\rho}^{(\b{\xi},h)}_t \hat{b}^{\xi,h}_t - \beta^{-1}\nabla \bar{\rho}^{(\b{\xi},h)}_t \right)dx.
    \end{aligned}
\end{equation*}
Note that both $b$ and $b^{\xi,h}$ have polynomial upper bounds.
By Young's inequality, Jensen's inequality, and the uniform moment bound in Lemma \ref{UITMB}, Claim 1 then follows.

For Claim 2 above, we control the negative Fisher information $-\mathcal{I}(\bar{\rho}^{\b{\xi}}_t)$ by the negative relative information $-\mathcal{I}(\bar{\rho}^{(\b{\xi},h)}_t|\pi) := -\int |\nabla \log \frac{\bar{\rho}^{(\b{\xi},h)}_t}{\pi}|^2 \bar{\rho}^{(\b{\xi},h)}_t dx$ via Young's inequality and obtain that
\begin{equation*}
    -\mathcal{I}(\bar{\rho}^{(\b{\xi},h)}_t) \leq -\frac{1}{2}\mathcal{I}(\bar{\rho}^{(\b{\xi},h)}_t|\pi) + \beta^2\mathbb{E}\left[|b(\bar{X}_t)|^2\Big|\mathcal{F}_{\xi}\right].
\end{equation*}
Then, similarly, using Young's inequality, Jensen's inequality,  the uniform moment bound in Lemma \ref{UITMB}, and since $\pi$ satisfies a log-Sobolev inequality with constant $\kappa=C^{LS}_\pi$ by Assumption \ref{SGLDA}, we have
\begin{equation*}
     \frac{d}{dt}\mathcal{H}(\bar{\rho}^{(\b{\xi},h)}_t|\pi) \leq -\frac{1}{4\beta \kappa} \mathcal{H}(\bar{\rho}^{(\b{\xi},h)}_t|\pi) + c_1'.
\end{equation*}
Therefore, Claim 2 holds due to Gr\"onwall's inequality.

Now, by Claims 1, 2 above, for any $T>0$ one has
\begin{equation*}
    \begin{aligned}
    &\quad \int_0^T e^{-A_0(T-s)} \mathcal{I}(\bar{\rho}^{(\b{\xi},h)}_s) ds\\ &\leq -\tilde{c}_0  \int_0^T e^{-A_0(T-s)}  \left(\frac{d}{ds}\mathcal{H}(\bar{\rho}^{(\b{\xi},h)}_s|\pi) \right)ds + \tilde{c}_1 \int_0^T e^{-A_0(T-s)}  ds
    \end{aligned}
\end{equation*}
Using integration by parts and the uniform bound of the relative entropy, one can obtain a bound independent of the choice of $T$ and $\b{\xi}$ for this integtral.
\end{proof}

\subsection{The backward conditional expectation}

The proof in this section is nearly the same as that of Lemmas A.2 - A.3 in \cite{li2022sharp}. The only difference during the analysis is that, we avoid using the boundedness condition of derivatives of the drift function $b$. Instead, we make use of their polynomial upper bounds in most places below. However, since the moment of T-SGLD is uniformly bounded, we can still obtain similar upper bounds as in \cite{li2022sharp}.

\begin{lemma}\label{integrationbyparts}
Recall the definitions of $X(t)$, $\bar{X}(t)$, $\xi_k$ and $\bar{\rho}$ in section \ref{sec:SGLD}. Suppose the assumptions used in Theorem \ref{thm:tsgld}  hold. Then
\begin{equation*}
    \mathbb{E}_{\xi_k}\mathbb{E}\Big[\Big|\mathbb{E}\Big[ \bar{X}(T_{k})-\bar{X}(t)\mid \bar{X}(t)\Big]\Big|^2 \mid \xi_k\Big]\leq Ch^{2}(1+\mathcal{I}(\bar{\rho}_{T_{k-1}}^{(h)})).
\end{equation*}
\end{lemma}

\begin{proof}

By Bayes' law, we have
\begin{equation*}
\begin{aligned}
     \mathbb{E} [\bar{X}(T_k) - \bar{X}(t) | \bar{X}(t)  = x,\xi_k ] &= \int (y-x)P(\bar{X}(T_k) = y | \bar{X}(t) = x,\xi_k) dy\\
     & = \int (y-x) \frac{\bar{\rho}(T_k)(y)P(\bar{X}(t) = x | \bar{X}(T_k) = y,\xi_k)}{\bar{\rho}_t^{(\xi_k,h)}(x)} dy.
\end{aligned}
\end{equation*}
Note that $P(\bar{X}(t) = x | \bar{X}(T_k) = y,\xi_k)$ is Gaussian-type, and consequently its derivative is also similar to the Gaussian form. Motovated by this, we split the term $ \mathbb{E} [\bar{X}(T_k) - \bar{X}(t) | \bar{X}(t) = x, \xi_k ]$ into three parts and use integration by parts to handle them.
In detail, let
\begin{equation*}
    \begin{aligned}
    y - x &= \left(I_d + (t-T_k)  \nabla b^{\xi_k,h} (y) \right) \cdot \left(y - x + (t-T_k) b^{\xi_k,h}(y) \right)\\
    & \quad- (t-T_k)  \nabla b^{\xi_k,h}(y) \cdot \left(y - x + (t-T_k) b^{\xi_k,h}(y) \right)\\
    & \quad- (t-T_k) \cdot b^{\xi_k,h}(y)\\
    & := a_1(x,y) - a_2(x,y) - a_3(x,y),
    \end{aligned}
\end{equation*}
and define
\begin{equation}\label{Iidef}
    J_i(x) := \mathbb{E}\left[a_i( \bar{X}(t),\bar{X}(T_k)) | \bar{X}(t) = x \right], \quad i=1,2,3.
\end{equation}

(a) For the term $J_1$, by definition,
\begin{multline*}
   \nabla_y P(\bar{X}(t) = x | \bar{X}(T_k) = y,\xi_k) = -\frac{\left(4\pi \beta^{-1} (t-T_k)\right)^{-\frac{d}{2}}}{2\beta^{-1}(t-T_k)}\\
   \exp\left(-\frac{|x-y-b^{\xi_k}(y)(t-T_k)|^2}{4\beta^{-1}(t-T_k)}\right)
   \left(I_d + (t-T_k)  \nabla b^{\xi_k,h} (y) \right) \cdot \left(y - x + (t-T_k) b^{\xi_k,h}(y) \right).
\end{multline*}
Then integration by parts gives
\begin{equation*}
    J_1(x) = 2\beta^{-1}(t-T_k)\int  \frac{\nabla_y \bar{\rho}_{T_k}(y)}{\bar{\rho}_t^{\xi_k,h}(x)} P(\bar{X}(t) = x | \bar{X}(T_k) = y, \xi_k) dy.
\end{equation*}
Applying Bayes' law again, we have
\begin{equation*}
    J_1(x) = 2\beta^{-1}(t-T_k) \int  \frac{\nabla_y \bar{\rho}_{T_k}(y)}{\bar{\rho}_{T_k}(y)} P(  \bar{X}(T_k) = y|\bar{X}(t) = x, \xi_k) dy.
\end{equation*}
By Jensen's inequality,
\begin{equation*}
\begin{aligned}
    \mathbb{E}\left[|J_1(\bar{X}_t)|^2\Big| \xi_k \right] &\leq 4\beta^{-2}(t-T_k)^2 \int \bar{\rho}_t^{\xi_k,h}(x) \int  \left|\frac{\nabla_y \bar{\rho}_{T_k} (y)}{\bar{\rho}_{T_k}(y)}\right|^2 P(  \bar{X}(T_k) = y|\bar{X}(t) = x ,\xi_k) dy dx\\
    & = 4\beta^{-2}(t-T_k)^2 \mathcal{I}(\bar{\rho}_{T_k}) := c(t-T_k)^2\mathcal{I}(\bar{\rho}_{T_k}).
\end{aligned}
\end{equation*}

(b) For the term $J_2$ and $J_3$, using the polynomial growth condition indicated in Lemma \ref{DPL2}, the control obtained in Lemma \ref{UITMB}, and H\"older's inequality, it is easy to obtain that
\begin{equation*}
    \mathbb{E}\left[|J_2(\bar{X}(t))|^2\Big| \xi_k\right]   \leq   c'(t-T_k)^3,
    \quad \mathbb{E}\left[|J_3(\bar{X}(t))|^2\Big| \xi_k\right] \leq c''(t-T_k)^2.
\end{equation*}
where $c'$, $c''$ are positive constants independent of $k$ and $\xi_k$.    

Hence, Young's inequality and the above estimates give the following
\begin{multline*}
\quad\int |\mathbb{E} [\bar{X}(T_k) - \bar{X}(t) | \bar{X}(t)  = x, \xi_k ]|^2 \bar{\rho}_t^{\xi_k}(x)dx \\
\leq 3\sum_{i=1}^3 \mathbb{E} \left[|J_i(\bar{X}(t))|^2\Big| \xi_k \right] \leq \tilde{c} h^2  \left( 1 + \mathcal{I}(\bar{\rho}(T_k))\right).
\end{multline*}
\end{proof}

\subsection{The random batch error estimate}\label{sec:RBerror}

We now focus on $\mathbb{E}_{\xi_k,\tilde{\xi}_k} \left[\int |b^{\xi_k}-b|^2 \frac{|\bar{\rho}_t^{\tilde{\xi}_k} - \bar{\rho}_t^{\xi_k}|^2}{\bar{\rho}_t^{\tilde{\xi}_k}} dx\right]$. Again, we use the same ideas from Lemmas A.4 - A.6 of \cite{li2022sharp}. We only provide a sketch of the proof here.

\begin{lemma}\label{lmm:i}
Recall the notations 
\begin{equation*}
    K_1 := \sqrt{\frac{\beta}{2}}\,\tilde{b}(y)(x-y),\quad \tilde{b}(y) := \sqrt{\frac{\beta}{2}} (b^{\xi_k,h}-b^{\tilde{\xi}_k,h})(y).
\end{equation*}
Then under the conditions of Theorem \ref{thm:tsgld}, there exists a positive constant $c$ independent of $k$ and $\xi_k$,$\tilde{\xi}_k$ such that for $h<1$ and $t\in [T_k,T_{k+1})$, 
\begin{equation*}
    \int \bar{\rho}_t^{\tilde{\xi}_k,h} (x) \left(\int K_1  P(\bar{X}(T_k) = y|\bar{X}'(t) = x,\xi_k,\tilde{\xi}_k) dy \right)^2dx \leq ch^{2-2\alpha} \left(1+\mathcal{I}(\bar{\rho}_{T_k})\right).
\end{equation*}
\end{lemma}

\begin{proof}
The technique here is very similar to the one we have used in the proof of Lemma \ref{integrationbyparts}. Again, we split the term $K_1$ into the following three parts:
\begin{equation*}
    \begin{aligned}
    \tilde{b}(y) \cdot (y - x) &= \tilde{b}(y) \cdot \left(I_d + (t-T_k)  \nabla b^{\tilde{\xi}_k,h} (y) \right) \cdot \left(y - x + (t-T_k) b^{\tilde{\xi}_k,h}(y) \right)\\
    & \quad- (t-T_k) \tilde{b}(y) \cdot \nabla b^{\tilde{\xi}_k,h}(y) \cdot \left(y - x + (t-T_k) b^{\tilde{\xi}_k,h}(y) \right)\\
    & \quad- (t-T_k) \tilde{b}(y) \cdot b^{\tilde{\xi}_k,h}(y)\\
    & := \bar{a}_1(x,y) - \bar{a}_2(x,y) - \bar{a}_3(x,y),
    \end{aligned}
\end{equation*}
and define
\begin{equation*}
    \bar{J_i}(x) := \mathbb{E}\left[\bar{a}_i(\bar{X}(T_k), \bar{X}'(t)) \Big| \bar{X}'(t) = x, \xi_k,\tilde{\xi}_k  \right], \quad i=1,2,3.
\end{equation*}
Similarly, for the $\bar{J_1}$ term, using Bayes' formula and integration by parts, we have
\begin{equation*}
    \begin{aligned}
    \bar{J_1}(x) &= 2\beta^{-1}\int \tilde{b}(y) \cdot \left(I_d + (t-T_k)  \nabla b^{\tilde{\xi}_k,h} (y) \right) \cdot \frac{\beta}{2}\left(y - x + (t-T_k) b^{\tilde{\xi}_k,h}(y) \right)\\
    &\quad\frac{\bar{\rho}_{T_k}(y)}{\bar{\rho}_t^{\tilde{\xi}_k,h}(x)}  P(\bar{X}'(t) = x|\bar{X}(T_k) = y ) dy\\
    &= 2\beta^{-1}(t-T_k)\int \left(\nabla \tilde{b}(y) + \tilde{b}(y) \frac{\nabla \bar{\rho}_{T_k}(y)}{\bar{\rho}_{T_k}(y)}\right) P(\bar{X}(T_k) = y| \bar{X}'(t) = x,\xi_k,\tilde{\xi}_k) dy. 
    \end{aligned}
\end{equation*}
Since $\tilde{b}$ has an $O(h^{-\alpha})$ upper bound, and $\nabla \tilde{b}$ has a polynomial upper bound, one has by Jensen's inequality again that
\begin{equation}\label{433}
     \beta^2\mathbb{E} \left[| \bar{J_1}(\bar{X}'(t)) |^2\Big| \xi_k,\tilde{\xi}_k\right] \leq ch^{2-2\alpha}\left(1+\mathcal{I}(\bar{\rho}_{T_k})\right),
\end{equation}
and the constant $c$ is independent of $k$, $\xi_k$ and $\tilde{\xi}_k$. Also, using exactly the same derivation for Lemma \ref{integrationbyparts} above as well as the $O(h^{-\alpha})$ upper bound for $\tilde{b}$, we obtain that
\begin{equation*}
   \mathbb{E}\left[|\bar{J}_2(\bar{X}'_t)|^2\Big| \xi_k,\tilde{\xi}_k\right] \leq ch^{3-2\alpha},
\quad
    \mathbb{E}\left[|\bar{J}_3(\bar{X}'_t)|^2\Big| \xi_k,\tilde{\xi}_k\right] \leq c  h^{2-2\alpha},
\end{equation*}
where the constant $c$ is independent of $k$, $\xi_k$ and $\tilde{\xi}_k$. 
\end{proof}

We also have the following estimate.
\begin{lemma}\label{lmm:ii}
Recall the notations 
\begin{equation*}
    K_2 := \left(-\sqrt{\frac{\beta}{2}}\,\tilde{b}(y) \cdot b^{\tilde{\xi}_k,h}(y)- \frac{1}{2} |\tilde{b}(y)|^2\right)(t-T_k),\quad \tilde{b}(y) := \sqrt{\frac{\beta}{2}} \left(b^{\tilde{\xi}_k,h} - b^{\xi_k,h}\right)(y).
\end{equation*}
Then under the conditions of Theorem 
\ref{thm:tsgld}, there exists a positive constant $c$ independent of $k$ and $\xi_k$ such that for $h<1$, $t\in [T_k,T_{k+1})$, 
\begin{equation*}
    \int \bar{\rho}_t^{\tilde{\xi}_k} (x) \left(\int K_2  P(\bar{X}(T_k) = y|\bar{X}'(t) = x,\xi_k,\tilde{\xi}_k) dy \right)^2dx \leq ch^2.
\end{equation*}
The bound is independent of $k$, $\tilde{\xi}_k$, $\xi_k$.
\end{lemma}

The result actually follows from the Jensen's inequality, 
\begin{equation}\label{eq450}
    \begin{aligned}
     &\quad\int \bar{\rho}_t^{\tilde{\xi}_k} (x) \left(\int K_2  P(\bar{X}_{T_k} = y|\bar{X}'_t = x) dy \right)^2dx\\
     &\leq (t-T_k)^2\mathbb{E}\left[\left|\sqrt{\frac{\beta}{2}}\,\tilde{b}(\bar{X}_{T_k})\cdot b^{\tilde{\xi}_k,h}(\bar{X}_{T_k}) + \frac{1}{2}|\tilde{b}(\bar{X}_{T_k})|^2\right|^2\Big| \xi_k,\tilde{\xi}_k\right]
    \end{aligned}
\end{equation}
and the polynomial bounds of $\tilde{b}$ and $b^{\tilde{\xi}_k,h}$.

\begin{lemma}\label{lmm:iii}
Recall the notations
$K_3 := e^z - 1 - z$,
with 
\begin{equation}\label{def_z}
    z := \frac{\beta}{2}\,\tilde{b}(y)(x-y-(t-T_k)b^{\tilde{\xi}_k}(y)) - \frac{1}{2} |\tilde{b}(y)|^2(t-T_k),
\end{equation}
and
\begin{equation*}
    \tilde{b}(y) := \sqrt{\frac{\beta}{2}}\left(b^{\tilde{\xi}_k} - b^{\xi_k}\right)(y).
\end{equation*}
Then under the conditions of Theorem 
\ref{thm:tsgld}, there exists a positive constant $c$ independent of $k$, $\tilde{\xi}_k$ and $\xi_k$ such that for $h<1$, $t\in [T_k,T_{k+1})$, it holds
\begin{equation*}
     \int \bar{\rho}_t^{(\tilde{\xi}_k,h)} (x) \left(\int K_3  P(\bar{X}(T_k) = y|\bar{X}'(t) = x,\xi_k,\tilde{\xi}_k) dy \right)^2dx \leq c  h^2.
\end{equation*}
\end{lemma}

The result here follows similarly by Jensen's inequality, 
\begin{equation*}
    \begin{aligned}
    \int \bar{\rho}_t^{(\tilde{\xi}_k,h)} (x) \left(\int K_3 P(\bar{X}_{T_k} = y|\bar{X}'_t = x,\xi_k,\tilde{\xi}_k) dy \right)^2dx
     \leq \mathbb{E}\left[\left|e^{\hat{Y}(t)} - 1 - \hat{Y}(t) \right|^2\Big|\xi_k,\tilde{\xi}_k\right],
    \end{aligned}
\end{equation*}
and then It\^o's formula. We omit the details. Here, we denote the process 
\begin{equation*}
\begin{split}
    \hat{Y}(t) &:= \sqrt{\frac{\beta}{2}}\tilde{b}(\bar{X}(T_k)) \cdot \left( \bar{X}'(t) -\bar{X}(T_k)-(t-T_k)b^{\tilde{\xi}_k}(\bar{X}(T_k)) \right) - \frac{1}{2} |\tilde{b}(\bar{X}(T_k))|^2 \Delta t\\
    &= -\frac{1}{2}|\tilde{b}(\bar{X}(T_k))|^2\Delta t +\tilde{b}(\bar{X}(T_k))\cdot \int_{T_k}^t dW.
\end{split}
\end{equation*}
with $\Delta t := t - T_k, \quad \forall t \in [T_k,T_{k+1}).$

\bibliographystyle{plain}
\bibliography{main}

\end{document}